\documentclass[final,leqno,onefignum,onetabnum]{siamltex1213}

\title{Analysis of a Time Multigrid Algorithm for DG-Discretizations in Time} 

\author{Martin J. Gander
\thanks{Section de Math\'{e}matiques
	2-4 rue du Li\`{e}vre, CP 64
	CH-1211 Gen\`{e}ve
(\email{martin.gander@unige.ch})}
\and
Martin Neum\"uller
\thanks{Inst. of Comp. Mathematics
        Altenbergerstr.~69
        4040 Linz
        Austria
(\email{martin.neumueller@jku.at})}
}

\usepackage{amssymb,latexsym,amsmath}

\newcommand{\IC}{\mathbb{C}}
\newcommand{\IN}{\mathbb{N}}
\newcommand{\IP}{\mathbb{P}}

\newcommand{\IR}{\mathbb{R}}

\newenvironment{variational}
{\begin{list}{}
             {\setlength{\leftmargin}{0.5cm}}
             \item[]
}
{\end{list}}

\newtheorem{remark}[theorem]{Remark}

\newcommand{\abs}[1]{\left\lvert{#1}\right\rvert}
\newcommand{\norm}[1]{{\left\lVert{#1}\right\rVert}} 
\newcommand{\spv}[2]{{\left({#1},{#2}\right)}}

\usepackage{bbm}
\renewcommand{\i}{\mathbbm{i}}

\usepackage{bm}
\renewcommand{\vec}[1]{\bm{#1}}

\DeclareMathOperator*{\argsup}{argsup}
\DeclareMathOperator*{\arginf}{arginf}

\usepackage{tikz}
\usepackage{pgfplots}
\usepgfplotslibrary{patchplots}

\usetikzlibrary{external}
\tikzset{external/system call={latex \tikzexternalcheckshellescape -halt-on-error
-interaction=batchmode -jobname "\image" "\texsource";
dvips -o "\image".ps "\image".dvi;
ps2eps "\image.ps"}}
\usepackage[subrefformat=parens,labelformat=parens]{subfig}

\usepackage{array}
\makeatletter
\newcommand{\thickhline}{%
    \noalign {\ifnum 0=`}\fi \hrule height 1pt
    \futurelet \reserved@a \@xhline
}
\newcolumntype{"}{@{\hskip\tabcolsep\vrule width 1pt\hskip\tabcolsep}}
\makeatother

\usepackage{multirow}
\usepackage{rotating}

\begin{document}
\maketitle
\slugger{mms}{sinum}{xx}{x}{x--x}%slugger should be set to mms, siap, sicomp, sicon, sidma, sima, simax, sinum, siopt, sisc, or sirev

\begin{abstract}
  We present and analyze for a scalar linear evolution model problem a
  time multigrid algorithm for DG-discretizations in time. We derive
  asymptotically optimized parameters for the smoother, and also an
  asymptotically sharp convergence estimate for the two grid cycle.
  Our results hold for any A-stable time stepping scheme and represent
  the core component for space-time multigrid methods for parabolic
  partial differential equations.  Our time multigrid method has
  excellent strong and weak scaling properties for parallelization in
  time, which we show with numerical experiments.
\end{abstract}

\begin{keywords}
Time parallel methods, multigrid in time, DG-discretizations, RADAU IA
\end{keywords}

\begin{AMS}
65N55, 65L60, 65F10
\end{AMS}

\pagestyle{myheadings}
\thispagestyle{plain}
\markboth{}{}

\section{Introduction}

The parallelization of algorithms for evolution problems in the time
direction is currently an active area of research, because today's
supercomputers with their millions of cores can not be effectively
used any more when only parallelizing the spatial directions.  In
addition to multiple shooting and parareal
  \cite{Lions2001, gander2007analysis, gander:2008:nca}, domain
decomposition and waveform relaxation \cite{Gander:1998:STC,
  Gander:2003:OSWW, Gander:2007:OSW}, and
direct time parallel methods \cite{maday2008parallelization, gander2013paraexp,
  Gander:2014:ADS}, multigrid methods in time are the fourth main
approach that can be used to this effect, see the overview
\cite{Gander:2014:50Y} and references therein. The
parabolic multigrid method proposed by Hackbusch in
\cite{Hackbusch:1984:PMG} was the first multigrid method in
space-time. It uses a smoothing iteration (e.g. Gauss-Seidel) over
many time levels, but coarsening is in general only possible in space,
since time coarsening might lead to divergence of the algorithm. The
multigrid waveform relaxation method proposed by Lubich and Ostermann
in \cite{Lubich:1987:MGD} is defined by applying a standard multigrid
method to the Laplace transform in time of the evolution problem. This
leads after backtransform to smoothers of waveform relaxation
type. The first complete space-time multigrid method that also allowed
coarsening in time was proposed by Horten and Vandewalle in
\cite{horton1995space}. It uses adaptive semi-coarsening in space or
time and special prolongation operators only forward in time.  The
analysis is based on Fourier techniques, and fully mesh independent
convergence can be obtained for F-cycles.  More recently, Emmett and
Minion proposed the Parallel Full Approximation Scheme in Space-Time
(PFASST), which is a non-linear multigrid method using a spectral
deferred correction iteration as the smoother, see
\cite{emmett2012toward}. This method has been successfully tested on
real problems, see for example
\cite{speck2012massively,speck2013multi}, but there is no convergence
analysis so far. A further time multigrid method can be found in
\cite{Falgout:2014:PTI}.

We present and analyze in this paper a new multigrid method in time
based on a block Jacobi smoother and standard restriction and
prolongation operators in time. This algorithm appeared for the first
time in the PhD thesis \cite{NeumuellerThesis2013}. To focus only on
the time direction, we consider for $T>0$ the one-dimensional model
problem
\begin{equation}\label{chap4_ordinaryDGL}
\begin{aligned}
 	\partial_t u(t) + u(t) &=  f(t) &\quad &\text{for } t \in (0,T), \\
	u(0) &= u_0,
\end{aligned}
\end{equation}
where $u_0 \in \IR$ and $f : (0,T) \rightarrow \IR$ are some given
data. In Section \ref{DiscretizationSec}, we present a general
Discontinuous Galerkin (DG) time stepping scheme for
(\ref{chap4_ordinaryDGL}), originally introduced by Reed and Hill
\cite{Reed1973} for neutron transport, see also \cite{Lasaint1974} for
ODEs and \cite{Delfour1981}. We also show their relation to classical
A-stable time stepping methods. In Section \ref{MGSec}, we then
present our time multigrid method for the DG time stepping scheme.
Section \ref{FourierSec} contains a Fourier mode analysis, and
we determine asymptotically the best choice of the smoothing
parameter, and the associated contraction estimate for the two grid
method. We give in Section \ref{NumericalSec} numerical
results which show both the strong and weak scalability of our time
multigrid method. We give an outlook on the applicability of our time
multigrid method to parabolic PDEs in Section \ref{ConclusionSec}.

\section{Discretization}\label{DiscretizationSec}

We divide the time interval $[0,T]$ into $N \in \IN$ uniform
subintervals $0 = t_0 < t_1 < \cdots < t_{N-1} < t_N = T$, $t_n = n
\,\tau$, with time step $\tau = \frac{T}{N}$, see Figure
\ref{chap4_timeSteppingScheme}. 
\begin{figure}%[htpb]
\includegraphics{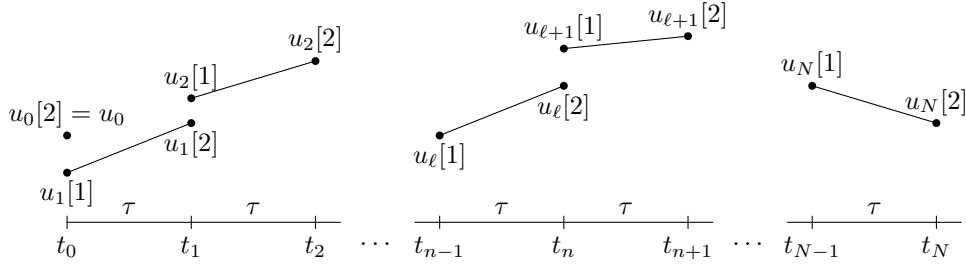}
\caption{DG time stepping scheme for $p_t=1$.}\label{chap4_timeSteppingScheme}
\end{figure}
By introducing the continuity condition $u_\tau^n(t_{n-1}) =
u_\tau^{n-1}(t_{n-1})$ in a weak sense, we obtain for $(t_{n-1},t_n)$
the discrete variational problem
\begin{variational}
  Find $u_\tau^n \in \IP^{p_t}(t_{n-1},t_n)$ such that for all 
  $v_\tau^n \in \IP^{p_t}(t_{n-1},t_n)$
  \begin{equation}\label{chap4_discreteVarFormODE}
    \begin{aligned}
     &-\int_{t_{n-1}}^{t_{n}} u_\tau^n(t) \partial_t v_\tau^n(t) \mathrm{d}t
  + u_\tau^n({t_{n}}) v_\tau^n({t_{n}}) + \int_{t_{n-1}}^{t_{n}} u_\tau^n(t) v_\tau^n(t)
  \mathrm{d}t \\
  &\qquad= \int_{t_{n-1}}^{t_{n}} f(t) v_\tau^n(t) \mathrm{d}t + 
  u_\tau^{n-1}({t_{n-1}}) v_\tau^n({t_{n-1}}).
    \end{aligned}
  \end{equation}
\end{variational}

%%%%%%%%%%%%%%%%%%%%%%%%%%%%%%%%%%%%%%%%%%%%%%%%%%%%%%%%%%%%%%%%%%%%%%%%%%%%%%%%
\subsection{Linear system}
%%%%%%%%%%%%%%%%%%%%%%%%%%%%%%%%%%%%%%%%%%%%%%%%%%%%%%%%%%%%%%%%%%%%%%%%%%%%%%%%

Using the basis functions
\begin{align}\label{chap4_labelBasisFunctions}
\IP^{p_t}(t_{n-1},t_n) = \mathrm{span}\{ \psi_\ell^n \}_{\ell=1}^{N_t}, \qquad
N_t = p_t+1,
\end{align}
the discrete variational problem (\ref{chap4_discreteVarFormODE}) is equivalent
to the linear system
\begin{align*}
   \left[ K_\tau + M_\tau\right] \vec{u}_n = \vec{f}_n + N_{\tau} \vec{u}_{n-1},
\end{align*}
with the matrices
\begin{align*}
 K_\tau[k,\ell] &:=  -\int_{t_{n-1}}^{t_{n}} \psi_\ell^n(t) \partial_t \psi_k^n(t) \mathrm{d}t
  + \psi_\ell^n({t_{n}}) \psi_k^n({t_{n}}),\\
 M_\tau[k,\ell] &:= \int_{t_{n-1}}^{t_{n}} \psi_\ell^n(t) \psi_k^n(t)
  \mathrm{d}t, \qquad N_\tau[k,\ell] := \psi_\ell^{n-1}({t_{n-1}}) \psi_k^n({t_{n-1}})
\end{align*}
for $k, \ell = 1,\ldots,N_t$. The right hand sides are given by
\begin{align*}
  \vec{f}_{n}[\ell] := \int_{t_{n-1}}^{t_{n}} f(t) \psi_\ell^n(t) \mathrm{d}t,
  \qquad \ell = 1,\ldots,N_t.
\end{align*}
On the time interval  $[t_{n-1},t_n]$ we then have the approximation
$u_\tau^n(t) = \sum_{\ell = 1}^{N_t} u_n[\ell] \psi_\ell^n(t)$, and
for the coefficients, we have to  solve the block triangular linear system
\begin{equation}\label{chap4_generalLinearSystemODE}
  \begin{pmatrix}
    K_\tau + M_\tau & & & \\
    -N_{\tau} & K_\tau + M_\tau & & \\
     &\ddots & \ddots & & \\
    & & -N_{\tau} & K_\tau + M_\tau
  \end{pmatrix}
  \begin{pmatrix}
    \vec{u}_1 \\
    \vec{u}_2 \\
    \vdots \\
    \vec{u}_N \\
  \end{pmatrix} = 
  \begin{pmatrix}
    \vec{f}_1 \\
    \vec{f}_2 \\
    \vdots \\
    \vec{f}_N \\
  \end{pmatrix}.
\end{equation}
Using for example constant polynomials, we simply obtain the backward
Euler scheme.

\subsection{General properties}

To study the properties of the discontinuous Galerkin discretization
(\ref{chap4_discreteVarFormODE}), we consider for a function $f:
(t_{n-1},t_n) \rightarrow \IR$ the Radau quadrature rule of order $2s-1$,
\begin{align*}
\int_{t_{n-1}}^{t_n} f(t) \mathrm{d}t \approx \tau \sum_{k=1}^s b_k f(t_{n-1} + c_k \tau),
\end{align*}
with the weights $b_k \in \IR_+$ and the integration points $c_1 = 0$ and
$c_2,\ldots,c_s \in [0,1]$, see~\cite{Hairer2010}.
\begin{theorem}\label{chap4_theorem1a}
  The discontinuous Galerkin approximation
  (\ref{chap4_discreteVarFormODE}) of the model problem
  (\ref{chap4_ordinaryDGL}) is equivalent to the $(p_t+1)$-stage
  implicit Runge-Kutta scheme RADAU IA, if the integral of the right
  hand side is approximated by the Radau quadrature of order $2
  p_t+1$.
\end{theorem}
\begin{proof}
 With this approximation, and using integration by parts, we obtain
 from (\ref{chap4_discreteVarFormODE}) the variational problem
\begin{variational}
  Find $u_\tau^n \in \IP^{p_t}(t_{n-1},t_n)$ such that for all 
  $v_\tau^n \in \IP^{p_t}(t_{n-1},t_n)$
  \begin{equation}\label{chap4_discreteVarForm2}
    \begin{aligned}
     &\int_{t_{n-1}}^{t_{n}} \partial_t u_\tau^n(t) v_\tau^n(t) \mathrm{d}t
  +  u_\tau^n({t_{n-1}}) v_\tau^n({t_{n-1}}) + \int_{t_{n-1}}^{t_{n}} u_\tau^n(t) v_\tau^n(t)
  \mathrm{d}t \\
  &\qquad= \tau \sum_{k=1}^{p_t+1} b_k f(t_{n-1} +c_k\tau) v_\tau^n(t_{n-1}+c_k
   \tau) +  u_\tau^{n-1}({t_{n-1}}) v_\tau^n({t_{n-1}}).
    \end{aligned}
  \end{equation}
\end{variational}
 The idea of the proof is to apply the discontinuous collocation
 method introduced in \cite{Hairer2010a} to the model problem
 (\ref{chap4_ordinaryDGL}). Let $c_1 = 0$ and $c_2,\ldots,c_{p_t+1}
 \in [0,1]$ be the integration points of the Radau quadrature of order
 $2 p_t+1$ with the weights $b_1,\ldots,b_{p_t+1} \in \IR \setminus
 \{0\}$. Then the discontinuous collocation method is given by
 \begin{variational}
   Find $w_\tau^n \in \IP^{p_t}(t_{n-1},t_n)$ such that for all $k =
   2,\ldots,p_t+1$
   \begin{equation}\label{chap4_discontinuousCollocation}
     \begin{aligned}
        w_\tau^n(t_{n-1}) - w_\tau^{n-1}(t_{n-1}) &=\tau b_1
        \left[ f(t_{n-1}) - \partial_t w_\tau^n(t_{n-1}) - w_\tau^n(t_{n-1}) \right],\\
        \partial_t w_\tau^n(t_{n-1}+c_k \tau) + w_\tau^n(t_{n-1}+c_k \tau) &= f(t_{n-1}+c_k\tau).
      \end{aligned}
   \end{equation}
 \end{variational}
 In \cite{Hairer2010a} it was shown that the discontinuous collocation
 method (\ref{chap4_discontinuousCollocation}) is equivalent to the
 $(p_t+1)$-stage implicit Runge-Kutta scheme RADAU IA. Hence it
 remains to show the equivalence of the discontinuous Galerkin method
 (\ref{chap4_discreteVarForm2}) to the discontinuous collocation
 method (\ref{chap4_discontinuousCollocation}). First, we observe that
 $\partial_t u_\tau^n v_\tau^n$ and $u_\tau^n v_\tau^n$ are
 polynomials of degree at most $2 p_t$. Therefore we can replace the
 integrals on the left hand side of (\ref{chap4_discreteVarForm2})
 with the Radau quadrature of order $2 p_t+1$, and obtain
 \begin{equation}\label{chap4_equationWithQuadrature}
 \begin{aligned}
   \tau \sum_{k=1}^{p_t+1} b_k \partial_t & u_\tau^n(t_{n-1}+c_k \tau)
   v_\tau^n(t_{n-1}+c_k \tau) + u_\tau^n(t_{n-1}) v_\tau^n(t_{n-1})\\
   &\qquad\qquad\qquad+ \tau \sum_{k=1}^{p_t+1} b_k u_\tau^n(t_{n-1}+c_k \tau)
   v_\tau^n(t_{n-1}+c_k \tau)\\
   &= \tau \sum_{k=1}^{p_t+1} b_k f(t_{n-1} +c_k\tau) v_\tau^n(t_{n-1}+c_k
   \tau) + u_\tau^{n-1}(t_{n-1}) v_\tau^n(t_{n-1}),
 \end{aligned}
 \end{equation}
 with $v_\tau^n \in \IP^{p_t}(t_{n-1},t_n)$. As test functions $v_\tau^n$ we
 consider the Lagrange polynomials
 \[ \ell_i^n(t) = \prod_{\stackrel{j=1}{j \neq i}}^{p_t+1} \frac{t-(t_{n-1}+c_j
   \tau)}{\tau(c_i-c_j)} \qquad \text{for } i = 1,\ldots,p_t+1. \]
 Hence we have $\ell_i^n(t_{n-1}+c_j\tau) = 0$ for $i \neq j$ and $
 \ell_i^n(t_{n-1}+c_i\tau) = 1$ for $i=1,\ldots,p_t+1$. First we use the test
 function $v_\tau^n = \ell_1^n$ in (\ref{chap4_equationWithQuadrature}) and
 obtain
  \begin{align*}
   \tau b_1 \partial_t u_\tau^n(t_{n-1}) + u_\tau^n(t_{n-1}) + \tau b_1
   u_\tau^n(t_{n-1}) = \tau b_1 f(t_{n-1}) + u_\tau^{n-1}(t_{n-1}).
 \end{align*}
 This implies that the solution $u_\tau^n$ of (\ref{chap4_discreteVarForm2})
 satisfies the first equation of (\ref{chap4_discontinuousCollocation}). For
 the test function  $v_\tau^n = \ell_k^n$, $k = 2,\ldots,p_t+1$ we further get
 \begin{align*}
   \tau b_k \partial_t u_\tau^n(t_{n-1}+c_k \tau) + \tau b_k
   u_\tau^n(t_{n-1}+c_k \tau) = \tau b_k f(t_{n-1} +c_k\tau).
 \end{align*}
 Dividing this equation by $\tau b_k \neq 0$ we see that the solution
 $u_\tau^n$ of the discontinuous Galerkin scheme
 (\ref{chap4_discreteVarForm2}) also satisfies the second equation of
 the discontinuous collocation method
 (\ref{chap4_discontinuousCollocation}). Hence the solution $u_\tau^n$
 of the discontinuous Galerkin scheme (\ref{chap4_discreteVarForm2})
 is a solution of the discontinuous collocation method
 (\ref{chap4_discontinuousCollocation}). The converse is proved by
 reverting the arguments.
\end{proof}

The RADAU IA scheme has been introduced in the PhD thesis
\cite{Ehle1969} in 1969, see also \cite{Chipman1971}.  From the proof
of Theorem \ref{chap4_theorem1a} we see that the jump of the discrete
solution at time $t_{n-1}$ is equal to the point wise error multiplied
by the time step size $\tau$ and the weight $b_1$, see
(\ref{chap4_discontinuousCollocation}). Hence the height of the jump
can be used as a simple error estimator for adaptive time stepping.

\begin{theorem}\label{chap4_theorem2a}
  For $s \in \IN$, the $s$-stage RADAU IA scheme is of order $2s-1$ and the
  stability function $R(z)$ is given by the $(s-1,s)$ subdiagonal Pad\'{e}
  approximation of the exponential function $e^z$. Furthermore the method is
  A-stable, i.e.
  \[ \abs{R(z)} < 1 \qquad \text{for } z \in \IC \text{ with }\Re(z) < 0. \]
\end{theorem}
\begin{proof}
  The proof can be found in \cite{Hairer2010}.
\end{proof}

%\begin{remark}
%  For $s \in \IN$, the first $(s-1,s)$ subdiagonal Pad\'{e} approximations
%  $P_{s-1,s}(z)$ of the exponential function $e^z$ are given by
%  \begin{alignat*}{2}
%     P_{0,1}(z) &= \frac{1}{1-z}, &\qquad  P_{1,2}(z) &= \frac{6 + 2 z}{6 -4z +
%       z^2}, \\ 
%     P_{2,3}(z) &= \frac{60+24z+3z^2}{60-36z+9z^2-z^3}, &\quad  P_{3,4}(z) &=
%     \frac{840+360z+60z^2+4z^3}{840 -480z + 120z^2-16z^3+z^4}. 
%  \end{alignat*}
%\end{remark}

\begin{corollary}\label{chap4_corollary1}
  The stability function $R(z)$ of the discontinuous Galerkin
  approximation with polynomial degree $p_t \in \IN$ is given by the
  $(p_t,p_t+1)$ subdiagonal Pad\'{e} approximation of the exponential function
  $e^z$. Furthermore the method is A-stable, i.e.
  \[ \abs{R(z)} < 1 \qquad \text{for } z \in \IC \text{ with }\Re(z) < 0. \]
\end{corollary}
\begin{proof}
  For the Dahlquist test equation $\partial_t u = \lambda u, \lambda \in \IC$
  we obtain by Theorem \ref{chap4_theorem1a} that the discontinuous Galerkin
  scheme is equivalent to the RADAU IA method. Hence the two methods have the
  same stability function $R(z)$. Applying Theorem \ref{chap4_theorem2a}
  completes the proof.
\end{proof}

%%%%%%%%%%%%%%%%%%%%%%%%%%%%%%%%%%%%%%%%%%%%%%%%%%%%%%%%%%%%%%%%%%%%%%%%%%%%%%%%
%%%%%%%%%%%%%%%%%%%%%%%%%%%%%%%%%%%%%%%%%%%%%%%%%%%%%%%%%%%%%%%%%%%%%%%%%%%%%%%%
\section{Multigrid method}\label{MGSec}
%%%%%%%%%%%%%%%%%%%%%%%%%%%%%%%%%%%%%%%%%%%%%%%%%%%%%%%%%%%%%%%%%%%%%%%%%%%%%%%%
%%%%%%%%%%%%%%%%%%%%%%%%%%%%%%%%%%%%%%%%%%%%%%%%%%%%%%%%%%%%%%%%%%%%%%%%%%%%%%%%

To apply multigrid to (\ref{chap4_generalLinearSystemODE}), we write
the linear system (\ref{chap4_generalLinearSystemODE}) using Kronecker
products,
\begin{align}\label{chap4_equationLinearSystemODE}
\left[ I_N \otimes (K_\tau + M_\tau) + U_N \otimes N_\tau \right] \vec{u} =:
\mathcal{L}_\tau \, \vec{u} = \vec{f},
\end{align}
with the matrix
\begin{align}
  U_N := \begin{pmatrix}
    0 & & & & \\
    -1 & 0 & & & \\
     & \ddots & \ddots & \\
     & & -1 & 0
  \end{pmatrix} \in \IR^{N \times N}.
\end{align}
% \marginpar{Ich stelle einfach die groben Freiheitsgrade \"{u}ber die feinen dar
% und diese Transformation ergibt dann die Prolongation bzw. Restriktion - ist das auch
% mit linearer Interpolation gemeint, wenn man h\"{0}here Polynomgrade verwendet? Ich habe
% deshalb die Referenzen der exakten Definitionen angegeben.}
We assume a nested sequence of decompositions
$\mathcal{T}_{N_L}$ with time step $\tau_L$ for $L
=0,\ldots,M_L$. We use standard restriction and
prolongation operators $\mathcal{R}^L$ and $\mathcal{P}^L$, see
(\ref{chap4_restrictionTime}) and (\ref{chap4_prolongationTime}) and 
$\nu \in \IN$ steps of a damped block Jacobi smoother 
\begin{align}\label{chap4_ODESmoother}
\quad\vec{u}^{\nu+1} = \vec{u}^\nu + \omega_t D_{\tau_L}^{-1} \left[
  \vec{f} -  \mathcal{L}_{\tau_L}\, \vec{u}^{\nu} \right],\quad \omega_t \in (0,2)
%=:\mathcal{S}_{\tau_L}^{1}(\vec{x}^k, \vec{f}),
\end{align}
with block diagonal matrix $D_{\tau_L} := \mathrm{diag}\{ K_{\tau_L} +
M_{\tau_L} \}_{n=1}^{N_L}$. For a given time step size $\tau_L$,
the error of the $\nu+1$st Jacobi iteration for $\nu \in \IN_0$ is
given by 
\begin{align}\label{chap4_errorSmoother}
  \vec{u} - \vec{u}^{\nu+1} =: \vec{e}^{\nu+1} = \left[I - \omega_t
    D_{\tau_L}^{-1} \mathcal{L}_{\tau_L} \right]^\nu \vec{e}^{\nu} =:
  \mathcal{S}_{\tau_L}^\nu \vec{e}^{\nu}.
\end{align}
The $k+1$st error of the two-grid cycle is given by
\begin{align}\label{chap4_errorTwoGridCycle}
  \vec{u} - \vec{u}^{k+1} =: \vec{e}^{k+1} = \mathcal{S}_{\tau_L}^{\nu_2}
  \left[ I - \mathcal{P}^L \mathcal{L}_{2 \tau_L}^{-1} \mathcal{R}^L
    \mathcal{L}_{\tau_L} \right] \mathcal{S}_{\tau_L}^{\nu_1} \vec{e}^k =:
  \mathcal{M}_{\tau_L} \vec{e}^k,
\end{align}
where we use the same symbol $\vec{e}$ also for this error to keep the notation simple.
To ensure asymptotically mesh independent convergence of the two-grid
cycle we need that the spectral radius of the iteration matrix
$\mathcal{M}_{\tau_L}$ is smaller than one, i.e.
\begin{align*}
  \varrho(\mathcal{M}_{\tau_L}) \leq q < 1,
\end{align*}
with a constant $q$ independent of the time step size. The computation
of the spectral radius for arbitrary two-grid iteration matrices is in
general not trivial, because the inverse of the coarse grid operator
$\mathcal{L}_{2\tau_L}$ is involved. We therefore transform the
equation $(\ref{chap4_equationLinearSystemODE})$ into the frequency
domain, where we apply the analysis based on exponential Fourier
modes. This type of analysis was introduced in \cite{Brandt1977}, and
can be made rigorous for model problems with periodic boundary
conditions, see \cite{Brandt1994}, and also \cite{Stueben1985,
  Wesseling2004, Trottenberg2001}. For general boundary conditions,
one can in general only get some insight into the local behavior of the
two-grid algorithm, and the method is called Fourier mode
analysis. In our case, we will see however that the Fourier mode analysis
gives parameter and contraction estimates of excellent quality.

For periodic boundary conditions the problem (\ref{chap4_ordinaryDGL})
changes to
\begin{align}\label{chap4_periodicBC}
  \partial_t u(t) + u(t) = f(t) \quad \text{for } t \in (0,t), \qquad u(0) =
  u(T).
\end{align}
For the discretization of the problem (\ref{chap4_periodicBC}) with a
discontinuous Galerkin time stepping method we therefore
have to solve a modified linear system
(\ref{chap4_equationLinearSystemODE}), i.e.
\begin{align}\label{chap4_equationLinearSystemODEPeriodic}
 \left[ I_N \otimes (K_\tau + M_\tau) + \tilde{U}_N \otimes N_\tau \right]
\vec{u} =: \tilde{\mathcal{L}}_\tau \, \vec{u} = \vec{f},
\end{align}
where the matrix $\tilde{U}_N$ is given by the circulant matrix
\begin{align}\label{chap4_circulantMatrix}
 \tilde{U}_N := \begin{pmatrix}
    0 & & & & -1 \\
    -1 & 0 & & & \\
     & \ddots & \ddots & \\
     & & -1 & 0
  \end{pmatrix} \in \IR^{N \times N}.
\end{align}

%%%%%%%%%%%%%%%%%%%%%%%%%%%%%%%%%%%%%%%%%%%%%%%%%%%%%%%%%%%%%%%%%%%%%%%%%%%%%%%%
%%%%%%%%%%%%%%%%%%%%%%%%%%%%%%%%%%%%%%%%%%%%%%%%%%%%%%%%%%%%%%%%%%%%%%%%%%%%%%%%
\section{Fourier mode analysis}\label{FourierSec}
%%%%%%%%%%%%%%%%%%%%%%%%%%%%%%%%%%%%%%%%%%%%%%%%%%%%%%%%%%%%%%%%%%%%%%%%%%%%%%%%
%%%%%%%%%%%%%%%%%%%%%%%%%%%%%%%%%%%%%%%%%%%%%%%%%%%%%%%%%%%%%%%%%%%%%%%%%%%%%%%%
We now use Fourier mode analysis to study the behavior of the
block Jacobi smoother and the two-grid cycle.
\begin{theorem}[Discrete Fourier transform]\label{chap4_theorem1}
  For $m \in \IN$ and $\vec{u} \in \IR^{2 m}$ we have
  \begin{align*}
    \vec{u} = \sum_{k = 1-m}^{m} \hat{u}_k \vec{\varphi}(\theta_k), \qquad
    \vec{\varphi}(\theta_k)[\ell]:= e^{\i \ell \theta_k},\quad \ell = 1,\ldots,
    2m,  \qquad \theta_k := \frac{k \pi}{m},
  \end{align*}
  with the coefficients
  \begin{align*}
    \hat{u}_k := \frac{1}{2m} \spv{\vec{u}}{\vec{\varphi}(-\theta_k)}_{\ell^2} =
    \frac{1}{2m} \sum_{\ell = 1}^{2m} \vec{u}[\ell]
    \vec{\varphi}(-\theta_k)[\ell], \qquad \text{for } k = 1-m,\ldots,m.
  \end{align*}
\end{theorem}
\begin{proof}
  The proof can be found for example in \cite[Theorem 7.3.1]{Wesseling2004}.
\end{proof}
\begin{definition}[Fourier modes, Fourier frequencies]
  Let $N_L \in \IN$. Then the vector valued function
  $\vec{\varphi}(\theta_k)[\ell]:= e^{\i \ell \theta_k}$, $\ell = 1,\ldots,N_L$
  is called Fourier mode with frequency
  \begin{align*}
  \theta_k \in \Theta_L := \left\{ \frac{2 k \pi}{N_L} : k = 1 -
    \frac{N_L}{2},\ldots,\frac{N_L}{2} \right\} \subset (-\pi,\pi].
  \end{align*}
  The frequencies $\Theta_L$ are further separated into low and high
  frequencies
  \begin{align*}
     \Theta_L^{\mathrm{low}} &:= \Theta_L \cap (-\frac{\pi}{2},\frac{\pi}{2}],\\
     \Theta_L^{\mathrm{high}} &:= \Theta_L \cap \left( (-\pi,-\frac{\pi}{2}]
       \cup (\frac{\pi}{2},\pi] \right) = \Theta_L \setminus
     \Theta_L^{\mathrm{low}}.
  \end{align*}
\end{definition}
We denote by $N_L \in \IN$ the number of time steps for the level
$L\in \IN_0$, and by $N_t = p_t +1 \in \IN$ the degrees of freedom with
respect to one time step, see also
(\ref{chap4_labelBasisFunctions}). The next lemma permits the
transform of a given vector corresponding to problem
(\ref{chap4_equationLinearSystemODE}) into the frequency domain.
\begin{lemma}\label{chap4_lemma1}
  The vector $\vec{u} = (\vec{u}_1, \vec{u}_2,\ldots,\vec{u}_{N_L})^\top \in
  \IR^{N_L\,N_t}$ for $N_{L-1}, N_t \in \IN$ and $N_L = 2 N_{L-1}$ can be written
  as
  \[ \vec{u} = \sum_{k=-N_{L-1}+1}^{N_{L-1}} \vec{\psi}^L(\theta_k, U) =
  \sum_{\theta_k \in \Theta_L} \vec{\psi}^L(\theta_k,U), \]
  with the vectors
  $\vec{\psi}_n^L(\theta_k, U) := U \vec{\Phi}_n^L(\theta_k)$ and
  $\vec{\Phi}_n^L(\theta_k)[\ell] := \vec{\varphi}(\theta_k)[n]$
  for $n=1,\ldots,N_L$ and $\ell = 1,\ldots,N_t$,
  and the coefficient matrix
  $U = \mathrm{diag}(\hat{u}_k[1],\ldots,\hat{u}_k[N_t]) \in \IC^{N_t \times N_t}$
  with the coefficients
  $\hat{u}_k[\ell] := \frac{1}{N_{L}} \sum_{i=1}^{N_L} u_{i}[\ell]
  \vec{\varphi}(-\theta_k)[i]$ for $k = 1 -N_{L-1},\ldots,N_{L-1}$.
\end{lemma}
\begin{proof}
  For a fixed index $\ell \in \{ 1,\ldots,N_t\}$ we apply Theorem
  \ref{chap4_theorem1} to the vector $\tilde{\vec{u}}_\ell \in
  \IR^{N_L}$ with $\tilde{\vec{u}}_\ell[n] := \vec{u}_n[\ell]$, $n =
  1,\ldots,N_L$. Now by using the definition of the coefficient $
  \hat{u}_k[\ell]$ and the definition of the vector
  $\vec{\psi}_n^L(\theta_k)$, the statement of the lemma follows with
  \begin{align*}
    u_n[\ell] &= \tilde{\vec{u}}_\ell[n] =\sum_{k=-N_{L-1}+1}^{N_{L-1}}
    \hat{u}_k[\ell] \vec{\varphi}(\theta_k)[n]
     = \sum_{k=-N_{L-1}+1}^{N_{L-1}} \hat{u}_k[\ell]\vec{\Phi}_n^L(\theta_k)[\ell]\\
     &= \sum_{k=-N_{L-1}+1}^{N_{L-1}} U[\ell,\ell]\vec{\Phi}_n^L(\theta_k)[\ell]
     = \sum_{k=-N_{L-1}+1}^{N_{L-1}} \vec{\psi}_n^L(\theta_k, U)[\ell] =
     \sum_{\theta_k \in \Theta_L} \vec{\psi}_n^L(\theta_k, U)[\ell].
  \end{align*}
\end{proof}

Note that in Lemma \ref{chap4_lemma1} the vector $\vec{\psi}^L =
\vec{\psi}^L (\theta_k, U)$ depends on the frequency $\theta_k \in
\Theta_L$ and on the coefficient matrix $U \in \IC^{N_t \times N_t}$,
where the coefficient matrix $U$ can be computed via the given vector
$\vec{u} = (\vec{u}_1, \vec{u}_2,\ldots,\vec{u}_{N_L})^\top$. In the
following we will study the mapping properties of the system matrix
$\mathcal{L}_{\tau_L}$ and the smoother $\mathcal{S}_{\tau_L}^\nu$
with respect to the vector $\vec{\psi}^L = \vec{\psi}^L (\theta_k,
U)$. Since the coefficient matrix $U$ will be fixed and since we have
to study the mapping properties of $\mathcal{L}_{\tau_L}$ and
$\mathcal{S}_{\tau_L}^\nu$ with respect to the frequencies $\theta_k
\in \Theta_L$, we will use the simpler notation $\vec{\psi}^L =
\vec{\psi}^L(\theta_k)$. The dependence of the vector $\vec{\psi}^L$
on the coefficient matrix $U$ is given in 
%Lemma \ref{chap4_lemma1} motivates the definition of the following space.
\begin{definition}[Fourier space]
  For $N_{L}, N_t \in \IN$ let the vector $\vec{\Phi}^L(\theta_k) \in \IC^{N_t
    N_L}$ be defined as in Lemma \ref{chap4_lemma1} with frequency $\theta_k
  \in \Theta_L$. Then we define the linear space of Fourier modes with
  frequency $\theta_k$ as
  \begin{align*}
    \Psi_L(\theta_k) &:= \mathrm{span}\left\{ \vec{\Phi}^L(\theta_k) \right\} \\
    &\phantom{:}= \left\{ \vec{\psi}^L(\theta_k) \in \IC^{N_t N_L} :
      \vec{\psi}_n^L(\theta_k) =
      U \vec{\Phi}_n^L(\theta_k),\; n = 1,\ldots,N_L \text{ and } U \in \IC^{N_t
        \times N_t} \right\}.
\end{align*}
\end{definition}

\subsection{Smoothing analysis}

To study the mapping properties of the system matrix
$\mathcal{L}_{\tau_L}$ and the smoother $\mathcal{S}_{\tau_L}^\nu$, we need
the following
\begin{lemma}\label{chap4_lemma2}
   For $N_L, N_t \in \IN$ let $\vec{\psi}^L(\theta_k) \in \Psi_L(\theta_k)$.
   Then we have for $n = 2,\ldots,N_L$ the shifting equality 
   $\vec{\psi}_{n-1}^L(\theta_k) = e^{-\i \theta_k} \vec{\psi}_n^L(\theta_k)$.
\end{lemma}
\begin{proof}
  Using the definition of the blockwise Fourier mode $\vec{\psi}^L(\theta_k)
  \in \Psi_L(\theta_k)$, we get the statement of the lemma for $n = 2,\ldots,
  N_L$ and $\ell = 1,\ldots,N_t$ with
  \begin{align*}
    \vec{\psi}_{n-1}^L(\theta_k)[\ell] &=  \sum_{i=1}^{N_t} U[\ell,i]
    \vec{\Phi}_{n-1}^L(\theta_k)[i] = \sum_{i=1}^{N_t} U[\ell,i]
    \vec{\varphi}(\theta_k)[n-1] = \sum_{i=1}^{N_t} U[\ell,i]
    e^{\i (n-1) \theta_k}\\ 
    &= e^{-\i \theta_k} \sum_{i=1}^{N_t} U[\ell,i] e^{\i n \theta_k} = e^{-\i
      \theta_k} \sum_{i=1}^{N_t} U[\ell,i]\vec{\varphi}(\theta_k)[n]\\
    &= e^{-\i \theta_k} \sum_{i=1}^{N_t} U[\ell,i]
    \vec{\Phi}_{n}^L(\theta_k)[i] = e^{-\i \theta_k}
    \vec{\psi}_{n}^L(\theta_k)[\ell].
  \end{align*}
\end{proof}

We can now obtain the Fourier symbol of the periodic system matrix $\tilde{\mathcal{L}}_{\tau_L}$.
\begin{lemma}\label{chap4_lemma3}
   For $N_L, N_t \in \IN$ let $\vec{\psi}^L(\theta_k) \in
   \Psi_L(\theta_k)$. Then for the system matrix $\tilde{\mathcal{L}}_{\tau_L}$ as
   defined in (\ref{chap4_equationLinearSystemODEPeriodic}) the Fourier symbol is
  \begin{align*}
    \big( \tilde{\mathcal{L}}_{\tau_L} \vec{\psi}^L(\theta_k) \big)_n = \left(
      K_{\tau_L} + M_{\tau_L} - e^{-\i \theta_k}
    N_{\tau_L}\right) \vec{\psi}_n^L(\theta_k) \quad \text{for } n=1,\ldots,N_L.
  \end{align*}
\end{lemma}
\begin{proof}
  Using the representation (\ref{chap4_equationLinearSystemODEPeriodic}) of
  the matrix $\tilde{\mathcal{L}}_{\tau_L}$, we get for a fixed but arbitrary
  $j = 1,\ldots,N_t$
  \begin{align*}
    \big( \tilde{\mathcal{L}}_{\tau_L}  \vec{\psi}^L(\theta_k) \big)_n[j] &=
    \sum_{m=1}^{N_L}
    \sum_{i=1}^{N_t} \big( I_{N_L}[n,m] (K_{\tau_L} + M_{\tau_L})[j,i] +
      \tilde{U}_{N_L}[n,m]N_{\tau_L}[j,i] \big)  \vec{\psi}_m^L(\theta_k)[i]\\
      &= \sum_{i=1}^{N_t}  (K_{\tau_L} + M_{\tau_L})[j,i]
      \vec{\psi}_n^L(\theta_k)[i] +
      \sum_{i=1}^{N_t} N_{\tau_L}[j,i] \sum_{m=1}^{N_L}
      \tilde{U}_{N_L}[n,m]\vec{\psi}_m^L(\theta_k)[i]\\
      &= \sum_{i=1}^{N_t}  (K_{\tau_L} + M_{\tau_L})[j,i]
      \vec{\psi}_n^L(\theta_k)[i] -
      \sum_{i=1}^{N_t} N_{\tau_L}[j,i] \vec{\psi}_{n-1}^L(\theta_k)[i]\\
      &= \sum_{i=1}^{N_t} \big( K_{\tau_L} + M_{\tau_L} - e^{-\i \theta_k}
      N_{\tau_L} \big) [j,i] \vec{\psi}_n^L(\theta_k)[i]\\
      &= \left( \big( K_{\tau_L} + M_{\tau_L} - e^{-\i \theta_k}
      N_{\tau_L} \big)  \vec{\psi}_n^L(\theta_k) \right) [j],
  \end{align*}
  where we used the definition of the matrix $\tilde{U}_{N_L}$ and Lemma
  \ref{chap4_lemma2}, assuming $n\neq 1$. For $n=1$, we observe that
  $$
    \sum_{m=1}^{N_L} \tilde{U}_{N_L}[n,m]\vec{\psi}_m^L(\theta_k)[i] =
    -N_{\tau_L}[j,i] \vec{\psi}_{N_L}^L(\theta_k)[i] =
    - e^{-\i \theta_k} N_{\tau_L} \vec{\psi}_n^L(\theta_k)[i],
  $$
 and hence we conclude that
  $
    \big( \tilde{\mathcal{L}}_{\tau_L} \vec{\psi}^L(\theta_k) \big)_n[j] = \left(\big( K_{\tau_L} +
    M_{\tau_L} - e^{-\i \theta_k} N_{\tau_L}\big) \vec{\psi}_n^L(\theta_k) \right)[j].
  $
\end{proof}

Lemma \ref{chap4_lemma3} shows that the periodic system matrix
$\tilde{\mathcal{L}}_{\tau_L}$ is a self-map on the Fourier space
$\Psi_L(\theta_k)$, i.e.  $\tilde{\mathcal{L}}_{\tau_L} : \Psi_L(\theta_k)
\rightarrow \Psi_L(\theta_k)$. 
This would not be the case for the system matrix
$\mathcal{L}_{\tau_L}$, but the two are closely related.

We next obtain the Fourier symbol of the periodic smoother
$\tilde{\mathcal{S}}_{\tau_L}^\nu := \left[I - \omega_t
    D_{\tau_L}^{-1} \tilde{\mathcal{L}}_{\tau_L} \right]^\nu$.
\begin{lemma}\label{chap4_lemma3a}
   For $N_L, N_t \in \IN$ let $\vec{\psi}^L(\theta_k) \in
   \Psi_L(\theta_k)$. Then for the smoother $\tilde{\mathcal{S}}_{\tau_L}^\nu$,
   we obtain for $\omega_t \in \IR$ the symbol
  \begin{align*}
    \big( \tilde{\mathcal{S}}_{\tau_L}^\nu \vec{\psi}^L(\theta_k) \big)_n = S_{\tau_L}(\theta_k,\omega_t)
    \vec{\psi}_n^L(\theta_k) \qquad \text{for } n = 1,\ldots, N_L,
  \end{align*}
  with the local iteration matrix
  \begin{align*}
    S_{\tau_L}(\theta_k,\omega_t) := (1-\omega_t) I_{N_t} + e^{-\i
          \theta_k} \omega_t (K_{\tau_L} + M_{\tau_L})^{-1}
      N_{\tau_L}.
  \end{align*}
\end{lemma}
\begin{proof}
  Let $\vec{\psi}^L(\theta_k) \in \Psi_L(\theta_k)$ and $\nu = 1$. Then, for $n =
  1,\ldots,N_L$ we obtain, using that $D_{\tau_L}^{-1}$ is a block
  diagonal matrix and applying Lemma \ref{chap4_lemma3}
  \begin{align*}
    \big( \tilde{\mathcal{S}}_{\tau_L}^1  \vec{\psi}^L(\theta_k)  \big)_n &= \big(
    \left( I_{N_L N_t} - \omega_t D_{\tau_L}^{-1} \tilde{\mathcal{L}}_{\tau_L}\right)
    \vec{\psi}^L(\theta_k)  \big)_n\\
  &=   \vec{\psi}_n^L(\theta_k)  - \omega_t (K_{\tau_L} + M_{\tau_L})^{-1} \big(
  \tilde{\mathcal{L}}_{\tau_L}   \vec{\psi}^L(\theta_k)  \big)_n\\
  &=   \vec{\psi}_n^L(\theta_k)  - \omega_t (K_{\tau_L} + M_{\tau_L})^{-1} \left(
    K_{\tau_L} + M_{\tau_L} - e^{-\i \theta_k} N_{\tau_L}\right)
  \vec{\psi}_n^L(\theta_k) \\
  &= \left( (1-\omega_t)
      I_{N_t} + e^{-\i \theta_k} \omega_t (K_{\tau_L} + M_{\tau_L})^{-1}
      N_{\tau_L}\right)  \vec{\psi}_n^L(\theta_k).
  \end{align*}
  For $\nu>1$ the statement follows simply by induction.
\end{proof}

To analyze the smoothing behavior of the damped block Jacobi smoother
$\tilde{\mathcal{S}}_{\tau_L}^\nu$, we have to estimate the spectral radius of
the
local iteration matrix
\[ 
  S_{\tau_L}(\theta_k,\omega_t) =  (1-\omega_t) I_{N_t} + e^{-\i \theta_k} \omega_t
  (K_{\tau_L} + M_{\tau_L})^{-1} N_{\tau_L}\in \IC^{N_t \times N_t}.
\]
Hence, we have to compute the eigenvalues of the matrix
$(K_{\tau_L} + M_{\tau_L})^{-1} N_{\tau_L}$.

\begin{lemma}\label{chap4_theorem2}
  For $\lambda \in \IC$ the eigenvalues of the matrix $(K_{\tau_L} - \lambda M_{\tau_L})^{-1} N_{\tau_L} \in
  \IC^{N_t \times N_t}$ are given by
  \begin{align*}
    \sigma((K_{\tau_L} - \lambda M_{\tau_L})^{-1} N_{\tau_L}) = \{ 0, R(\lambda \tau_L) \},
  \end{align*}
  where $R(z)$ is the A-stability function of the given discontinuous Galerkin time stepping scheme.
\end{lemma}
\begin{proof}
  First we notice that the eigenvalues of the matrix $(K_{\tau_L} -
  \lambda M_{\tau_L})^{-1} N_{\tau_L}$ are independent of the basis $\{ \psi_k
  \}_{k=1}^{N_t}$ which is used to compute the matrices $K_{\tau_L},
  M_{\tau_L}$ and $N_{\tau_L}$. Hence we can use basis functions $\{
  \psi_k \}_{k=1}^{N_t}$ where the eigenvalues of the matrix $(K_{\tau_L} - \lambda
  M_{\tau_L})^{-1} N_{\tau_L}$ are easy to compute, i.e. 
  polynomials $\psi_k \in \IP^{p_t}(0,\tau_L)$ with the property
  \begin{align*}
    \psi_k(\tau_L) = \begin{cases} 1 & k=1,\\ 0 & k \neq 1 \end{cases} \qquad
    \text{for } k = 1, \ldots, N_t.
  \end{align*}
  To study the A-Stability of the discontinuous Galerkin
  discretization, we consider for $\lambda \in \IC$ the model problem
  \[ \partial_t u(t) = \lambda u(t), \quad t \in (0,\tau_L) \quad \text{and}
  \quad u(0) = u_0. \]
  This leads to the linear system
  \[ \left(K_{\tau_L} - \lambda M_{\tau_L}\right) \vec{u}_1 = u_0 N_{\tau_L}
  \vec{v}, \]
  with the vector $\vec{v}[1] = 1$ and $\vec{v}[k] = 0$ for $k =
  2,\ldots,N_t$ and with the solution vector $\vec{u}_1 \in \IR^{N_t}$ for the
  first step. Therefore the value at the endpoint $\tau_L$ of the discrete
  solution is given by
  \[   u_1 = u_0 \vec{v}^\top \left(K_{\tau_L} - \lambda
    M_{\tau_L}\right)^{-1} N_{\tau_L} \vec{v} \in \IR. \]
  Hence the stability function $R(z)$ with $z = \lambda \tau_L$ is given by
  \begin{align}\label{chap4_theorem2_equ2}
    R(z(\lambda,\tau_L)) = R(\lambda\tau_L) = \vec{v}^\top \left(K_{\tau_L} - \lambda
    M_{\tau_L}\right)^{-1} N_{\tau_L} \vec{v}.
  \end{align}
  Since the matrix $N_{\tau_L}$ has rank one, only one eigenvalue can
  be nonzero and with (\ref{chap4_theorem2_equ2}), it is easy to see that 
  this eigenvalue is given by $R(\lambda\tau_L)$.
\end{proof}

Lemma \ref{chap4_theorem2} holds for any one step method. Hence a
one step method is A-stable if and only if
\begin{align*}
    \abs{R(z(\lambda,\tau_L))}= \varrho(\left(K_{\tau_L} - \lambda
      M_{\tau_L}\right)^{-1} N_{\tau_L}) < 1 \qquad \text{for all } z \in \IC
    \text{ with }\Re(z) < 0.
\end{align*}
Now we are able to compute the spectral radius of the local iteration matrix
$S_{\tau_L}(\theta_k,\omega_t) = (1-\omega_t) I_{N_t} + e^{-\i \theta_k} \omega_t (K_{\tau_L} + M_{\tau_L})^{-1}
N_{\tau_L} \in \IC^{N_t \times N_t}$.

\begin{lemma}\label{chap4_lemma4}
  Let $p_t \in  \IN_0$. Then for the smoother $\tilde{\mathcal{S}}_{\tau_L}^\nu$, the
  spectral radius of the local iteration matrix $S_{\tau_L}(\theta_k,\omega_t) = (1-\omega_t) I_{N_t} +
  e^{-\i \theta_k} \omega_t (K_{\tau_L} + M_{\tau_L})^{-1} N_{\tau_L}$ is given by
  \begin{align*}
    \varrho\left( S_{\tau_L}(\theta_k,\omega_t)\right) = \max\left\{\abs{1-\omega_t},
    \hat{S}(\omega_t, \alpha(\tau_L),
      \theta_k)\right\},
  \end{align*}
  with
 \begin{align*}
   \left(\hat{S}(\omega_t, \alpha, \theta_k)\right)^2 := (1-\omega_t)^2 + 2
   \omega_t (1 - \omega_t) \alpha \cos(\theta_k) + \alpha^2 \omega_t^2,
 \end{align*}
  where $\alpha = \alpha(t)$ is the  $(p_t,p_t+1)$ subdiagonal Pad\'{e}
  approximation of the exponential function $e^{-t}$.
\end{lemma}
\begin{proof}
   Since $I_{N_t}$ is the identity matrix, the eigenvalues of the local iteration
   matrix $S_{\tau_L}(\theta_k,\omega_t)$ are given by
   \begin{align*}
     \sigma(S_{\tau_L}(\theta_k,\omega_t)) = 1 - \omega_t + e^{-\i \theta_k} \omega_t  \sigma((K_{\tau_L} +
     M_{\tau_L})^{-1} N_{\tau_L}).
   \end{align*}
   With Theorem \ref{chap4_theorem2} we are now able to compute the spectrum of
   the iteration matrix $S_{\tau_L}(\theta_k,\omega_t)$,
   \begin{align*}
     \sigma(S_{\tau_L}(\theta_k,\omega_t)) = \left\{1 - \omega_t, 1 - \omega_t + e^{-\i \theta_k} \omega_t
       \alpha(\tau_L) \right\}.
   \end{align*}
   Hence we obtain the spectral radius
   \[ \varrho(S_{\tau_L}(\theta_k,\omega_t)) = \max \left\{ \abs{1 - \omega_t}, \abs{1 - \omega_t +
       e^{-\i \theta_k} \omega_t \alpha(\tau_L)} \right\}. \]
   Simple calculations lead to
   \begin{align*}
     \abs{1 - \omega_t +
       e^{-\i \theta_k} \omega_t \alpha(\tau_L)}^2 = (1-\omega_t)^2 + 2
   \omega_t (1 - \omega_t) \alpha(\tau_L) \cos(\theta_k) + (\alpha(\tau_L))^2
   \omega_t^2,
   \end{align*}
   which completes the proof.
\end{proof}
 
To proof the convergence of the block Jacobi smoother introduced in
(\ref{chap4_ODESmoother}), we will estimate the spectral radius of the local
iteration matrix $S_{\tau_L}(\theta_k,\omega_t) \in \IC^{N_t \times N_t}$.
\begin{lemma}\label{chap4_lemma5}
   Let $p_t \in \IN_0$ and $\omega_t \in (0,1]$, then the
  spectral radius of the local iteration matrix $S_{\tau_L}(\theta_k,\omega_t)
  = (1-\omega_t) I_{N_t} + e^{-\i \theta_k} \omega_t (K_{\tau_L} + M_{\tau_L})^{-1} N_{\tau_L}$
  is strictly bounded by one, i.e.
   \begin{align*}
    \varrho\left( S_{\tau_L}(\theta_k,\omega_t) \right) < 1.
  \end{align*}
\end{lemma}
\begin{proof}
  In view of Lemma \ref{chap4_lemma4} we have to estimate the function
  \[ \max\left\{\abs{1-\omega_t},\hat{S}(\omega_t, \tau_L,
    \theta_k)\right\}. \]
  For $\omega_t \in (0,1]$ we clearly have that $\abs{1-\omega_t} <
  1$. For $\hat{S}(\omega_t, \tau_L,  \theta_k)$ we estimate
  \begin{align*}
    \abs{\left(\hat{S}(\omega_t, \tau_L, \theta_k)\right)}^2 &= \abs{
      (1-\omega_t)^2 + 2
   \omega_t (1 - \omega_t) \alpha(\tau_L) \cos(\theta_k) + (\alpha(\tau_L))^2
   \omega_t^2}\\
   &\leq (1-\omega_t)^2 + 2 \omega_t (1 - \omega_t) \abs{\alpha(\tau_L)} +
   \abs{\alpha(\tau_L)}^2 \omega_t^2.
  \end{align*}
  Since $\alpha(\tau_L) = R(-\tau_L)$ is the A-stability function for
  $z = -\tau_L$, see Theorem \ref{chap4_theorem2}, and using the fact
  that the discontinuous Galerkin scheme is A-stable, see Corollary
  \ref{chap4_corollary1}, we have $\abs{\alpha(\tau_L)} < 1$ for
  $\tau_L > 0$. Hence we obtain the statement of this lemma with
  \begin{align*}
    \abs{\left(\hat{S}(\omega_t, \tau_L, \theta_k)\right)}^2 
   < (1-\omega_t)^2 + 2 \omega_t (1 - \omega_t) + \omega_t^2
   = \left( 1 -\omega_t + \omega_t \right)^2 = 1.
  \end{align*}
\end{proof}
\begin{theorem}\label{chap4_theorem3}
  For any damping parameter $\omega_t \in (0,1]$, the block Jacobi
    smoother introduced in (\ref{chap4_ODESmoother}) converges for any
    initial guess $\vec{u}^0$ to the exact solution of
    $\mathcal{L}_{\tau_L} \vec{u} = \vec{f}$.
\end{theorem}
\begin{proof}
  For an arbitrary but fixed $n \in \{ 1, \ldots, N_L \}$, the $n$-th
  error component $\vec{e}_n^\nu$ of the $\nu$-th damped block Jacobi
  iteration is given by
  \begin{align*}
    \vec{e}_n^\nu = \left(\mathcal{S}_{\tau_L}^\nu \vec{e}^0\right)_n
    =\left(\mathcal{S}_{\tau_L}^\nu\left( \sum_{\theta_k \in \Theta_L}
         \vec{\psi}^L(\theta_k)\right) \right)_n,
  \end{align*}
  with the initial error  $\vec{e}^0 = \vec{u} - \vec{u}^0$, 
  which we transformed into the
      frequency domain by applying Lemma \ref{chap4_lemma1}.
  The Fourier vectors $\vec{\psi}^L(\theta_k)$, $\theta_k
      \in \Theta_L$ depend on the  constant coefficient matrix $U = U(\vec{e}_0)
      \in \IC^{N_t \times N_t}$ resulting from the initial vector $\vec{e}_0$. Since
      $\mathcal{S}_{\tau_L}^\nu$ is a linear operator, we
      have, using Lemma \ref{chap4_lemma3a},
  \begin{align*}
    \vec{e}_n^\nu = \sum_{\theta_k \in \Theta_L} \left(\mathcal{S}_{\tau_L}^\nu
      \vec{\psi}^L(\theta_k) \right)_n
    %&=  \sum_{\theta_k \in \Theta_L} \left(  (1-\omega_t)
    %  I_{N_t} + e^{-\i \theta_k} \omega_t (K_{\tau_L} + M_{\tau_L})^{-1}
    %  N_{\tau_L} \right) \vec{\psi}_n^L(\theta_k)
    = \sum_{\theta_k \in \Theta_L} (S_{\tau_L}(\theta_k,\omega_t))^\nu \vec{\psi}_n^L(\theta_k).
  \end{align*}
  Now the spectral radius $ \varrho\left(S_{\tau_L}(\theta_k,\omega_t)\right)$ is
  strictly smaller than one, see Lemma \ref{chap4_lemma5}, and we
  conclude that $(S_{\tau_L}(\theta_k,\omega_t))^\nu \rightarrow 0$ as $\nu \rightarrow \infty$.
  This implies that the $n$-th component $\vec{e}_n^\nu$ of the $\nu$-th
  Jacobi iteration converges to zero as $\nu$ tends to infinity, i.e.
  \begin{align*}
     \vec{e}_n^\nu \rightarrow \mathbf{0} \qquad \text{for } \nu \rightarrow
     \infty.
  \end{align*}
  Hence $\vec{u}^\nu \rightarrow \vec{u}$ as the number of iterations $\nu$
  tends to infinity.
\end{proof}

In Theorem \ref{chap4_theorem3} the convergence of the damped
block Jacobi smoother with respect to the blocks is proven for
$\omega_t \in (0,1]$. A simpler approach would be
  to directly compute the spectral radius of the iteration matrix
  \begin{align*}
    \mathcal{S}_{\tau_L} = \begin{pmatrix}
    (1 - \omega_t)I_{N_t} & & & \\
    \omega_t(K_{\tau_L} + M_{\tau_L})^{-1}N_{\tau_L} & (1 - \omega_t)I_{N_t} & & \\  
    & \ddots & \ddots & \\
    & & \omega_t(K_{\tau_L} + M_{\tau_L})^{-1}N_{\tau_L} & (1 - \omega_t)I_{N_t}
  \end{pmatrix},
  \end{align*}
  which simply is $\varrho(\mathcal{S}_{\tau_L}) = \abs{1 -
    \omega_t}$.  Hence the damped block Jacobi smoother converges also
  for a damping parameter $\omega_t \in (0,2)$.  Choosing a damping
  parameter $\omega_t \in (1,2)$ leads indeed also to a
  convergent smoother, but not to a uniformly convergent
  one. This means that the error can grow for some blocks if we use a
  damping parameter $\omega_t \in (1,2)$, and one has to
    be careful using the spectral radius as a criterion in these
    highly non-symmetric cases. For a good smoother, we have to
  use a damping parameter $\omega_t \in (0,1]$.

For a good multigrid scheme, we need that the smoother reduces
the error in the high frequencies
$\Theta^{\mathrm{high}}$ efficiently. Theorem \ref{chap4_theorem3}
motivates 
\begin{definition}[Asymptotic smoothing factor]
  For the damped block Jacobi iteration introduced in
  (\ref{chap4_ODESmoother}), we define the asymptotic smoothing factor as
  \begin{align*}
    \mu_S := \max\left\{ \varrho\left(S_{\tau_L}(\theta_k,\omega_t)\right) : \theta_k \in
      \Theta_L^{\mathrm{high}}
    \text{ and } n \in \{1,\ldots,N_L\} \right\}
  \end{align*}
  with
  \begin{align*}
     S_{\tau_L}(\theta_k,\omega_t) = (1-\omega_t) I_{N_t} + e^{-\i \theta_k} \omega_t (K_{\tau_L} +
     M_{\tau_L})^{-1} N_{\tau_L}.
  \end{align*}
\end{definition}
To analyze the smoothing behavior, we will need
\begin{lemma}\label{chap4_lemma6}
  Let $\alpha \in \IR$ with $\alpha \geq -1$. Then for the function
  \begin{align*}
   \left(\hat{S}(\omega_t, \alpha, \theta_k)\right)^2 = (1-\omega_t)^2 + 2
   \omega_t (1 - \omega_t) \alpha \cos(\theta_k) + \alpha^2 \omega_t^2,
 \end{align*}
 the min-max principle
  \begin{align*}
    \inf_{\omega_t \in (0,1]} \sup_{\theta_k \in [\frac{\pi}{2},\pi]}
    \hat{S}(\omega_t, \alpha, \theta_k) = \begin{cases}
      \frac{\alpha}{\sqrt{1+\alpha^2}} & \alpha \geq 0,\\ \abs{\alpha} &
     \alpha < 0, \end{cases} \in [0,1]
  \end{align*}
  holds with the asymptotically optimal parameter
  \begin{align*}
    \omega_t^\ast = \begin{cases} \frac{1}{1+\alpha^2} & \alpha \geq 0,\\ 1
      &\alpha < 0 \end{cases} \quad \text{and} \quad \theta^\ast
    = \begin{cases} \frac{\pi}{2} & \alpha \geq 0,\\ \pi & \alpha <
      0. \end{cases}
  \end{align*}
\end{lemma}
\begin{proof}
  Since $\hat{S}(\omega_t, \alpha, \theta_k) \geq 0$, we will study the function
  \[ \left(\hat{S}(\omega_t, \alpha, \theta_k)\right)^2 = (1-\omega_t)^2 + 2
   \omega_t (1 - \omega_t) \alpha \cos(\theta_k) + \alpha^2 \omega_t^2. \]
  For $\omega_t \in (0,1]$, only the terms with $\alpha$ and
  $\cos(\theta_k)$ can become negative. We thus consider first the case 
  $\alpha \geq 0$. We then simply have
  \[ 
    \argsup_{\theta_k \in [\frac{\pi}{2},\pi]} \hat{S}(\omega_t, \alpha,
    \theta_k) = \frac{\pi}{2} \qquad \text{for } \omega_t \in (0,1], 
   \]
  which leads to 
  \begin{align*}
   \inf_{\omega_t \in (0,1]} \sup_{\theta_k \in [\frac{\pi}{2},\pi]}
    \hat{S}(\omega_t, \alpha, \theta_k) = \inf_{\omega_t \in (0,1]}
    \hat{S}(\omega_t, \alpha, \frac{\pi}{2}).
  \end{align*}
  Since
   $\left(\hat{S}(\omega_t, \alpha, \frac{\pi}{2})\right)^2 = (1-\omega_t)^2 +
     \alpha^2 \omega_t^2$,
  we find that
  \[ \arginf_{\omega_t \in (0,1]} \hat{S}(\omega_t, \alpha,
  \frac{\pi}{2}) = \frac{1}{1+\alpha^2} \quad \text{and} \quad
  \hat{S}(\frac{1}{1+\alpha^2}, \alpha,
  \frac{\pi}{2}) = \frac{\alpha}{\sqrt{1+\alpha^2}}. \]
  For the case $\alpha < 0$ we have
   \[ \argsup_{\theta_k \in [\frac{\pi}{2},\pi]} \hat{S}(\omega_t, \alpha,
   \theta_k) = \pi \qquad \text{for } \omega_t \in (0,1]. \]
  Because of
  \begin{align*}
    \left(\hat{S}(\omega_t, \alpha, \pi)\right)^2 = (1-\omega_t)^2 - 2 \omega_t
    (1 - \omega_t) \abs{\alpha} + \abs{\alpha}^2 \omega_t^2 = \left( 1 -
      \omega_t (1 + \abs{\alpha}) \right)^2,
  \end{align*}
  we find that
  \begin{align*}
    \arginf_{\omega_t \in (0,1]} \hat{S}(\omega_t, \alpha,
  \pi) = 1 \quad \text{and} \quad \hat{S}(1, \alpha, \pi) = \abs{\alpha},
  \end{align*}
  which completes the proof.
\end{proof}

The next lemma shows that the asymptotic smoothing factor $\mu_S$ is
strictly bounded by $\frac{1}{\sqrt{2}}$, if we use the optimal
damping parameter $\omega_t^\ast = \omega_t^\ast(\tau_L)$.
\begin{lemma}\label{chap4_lemma7}
  For the optimal choice of the damping parameter
  \begin{align*}
    \omega_t^\ast(\tau_L) := \begin{cases}
      \frac{1}{1+\left(\alpha(\tau_L)\right)^2} &
      \alpha(\tau_L) \geq 0,\\ 1
      &\alpha(\tau_L) < 0 \end{cases}
  \end{align*}
  the smoothing factor $\mu_S$ of the damped block Jacobi iteration
  (\ref{chap4_ODESmoother}) satisfies
  $\mu_S \leq \frac{1}{\sqrt{2}}$. 
\end{lemma}
\begin{proof}
  In view of Lemma \ref{chap4_lemma4} we have to estimate
  \begin{align*}
    \max_{\theta_k \in \Theta_L^{\mathrm{high}}}\left\{\abs{1-\omega_t^\ast},
      \hat{S}(\omega_t^\ast, \alpha(\tau_L), \theta_k)\right\}
  \end{align*}
  with
 \begin{align*}
   \left(\hat{S}(\omega_t^\ast, \alpha, \theta_k)\right)^2 =
   (1-\omega_t^\ast)^2 + 2
   \omega_t^\ast (1 - \omega_t^\ast) \alpha \cos(\theta_k) + \alpha^2
   (\omega_t^\ast)^2.
 \end{align*}
 Since $\hat{S}(\omega_t^\ast, \alpha, \theta_k)$ is symmetric with respect to
 the frequencies $\theta_k$, we only have to estimate the function
 $\hat{S}(\omega_t^\ast, \alpha, \theta_k)$ for the frequencies $\theta_k \in
 \Theta_L^{\mathrm{high}} \cap [\frac{\pi}{2}, \pi].$ Applying Lemma
 \ref{chap4_lemma6} for $\alpha = \alpha(\tau_L)$ gives the estimate
 \begin{align}\label{chap4_proofLemma7Equ1}
   \max_{\theta_k \in  \Theta_L^{\mathrm{high}}} \hat{S}(\omega_t^\ast, \alpha(\tau_L),
   \theta_k) \leq \sup_{\theta_k \in [\frac{\pi}{2},\pi]}
    \hat{S}(\omega_t^\ast, \alpha(\tau_L), \theta_k) = \begin{cases}
      \frac{\alpha(\tau_L)}{\sqrt{1+\left(\alpha(\tau_L)\right)^2}} &
      \alpha(\tau_L) \geq 0,\\ \abs{\alpha(\tau_L)} &
     \alpha(\tau_L) < 0. \end{cases}
 \end{align}
 Since $\alpha(\tau_L)$ is the $(p_t,p_t+1)$ subdiagonal Pad\'{e}
 approximation of the exponential function, see Lemma \ref{chap4_lemma4}, we
 have
 \[ -0.0980762 \approx \frac{5-3 \sqrt{3}}{2} \leq \alpha(\tau_L) \leq 1 \qquad
 \text{for } \tau_L \geq 0. \]
Combining this estimate with the results of (\ref{chap4_proofLemma7Equ1}) yields
\begin{align*}
   \max_{\theta_k \in  \Theta_L^{\mathrm{high}}} \hat{S}(\omega_t^\ast, \alpha,
   \theta_k) \leq \begin{cases}
      \frac{1}{\sqrt{2}} & \alpha \geq 0,\\ \frac{3 \sqrt{3}-5}{2} &
     \alpha < 0, \end{cases} \leq  \frac{1}{\sqrt{2}}.
\end{align*}
Simple calculations show that 
\[ \sup_{\theta_k \in [\frac{\pi}{2},\pi]}
    \hat{S}(\omega_t^\ast, \alpha, \theta_k) \geq \abs{1-\omega_t^\ast}, \]
which completes the proof.
\end{proof}

Because $\alpha(\tau_L)$ is the $(p_t,p_t+1)$ subdiagonal Pad\'{e}
approximation of the exponential function $e^{-t}$, we have that
$\alpha(\tau_L) \rightarrow 1$ as $\tau_L \rightarrow 0$, and hence
$\omega^\ast \approx \frac{1}{2}$ for $\tau_L$ close to zero, see 
Figure \ref{chap4_plot1_best_damping}. It turns out that the estimate
of Lemma \ref{chap4_lemma7} also holds for a uniform damping
parameter $\omega^\ast = \frac{1}{2}$. But for large time steps
$\tau_L$, better smoothing behavior is obtained when the optimal
damping parameter $\omega^\ast = \omega^\ast(\tau_L)$ as given in
Lemma \ref{chap4_lemma7} is used.
% \marginpar{Could one make the legends and texts in the figure a bit bigger?}
\begin{figure}
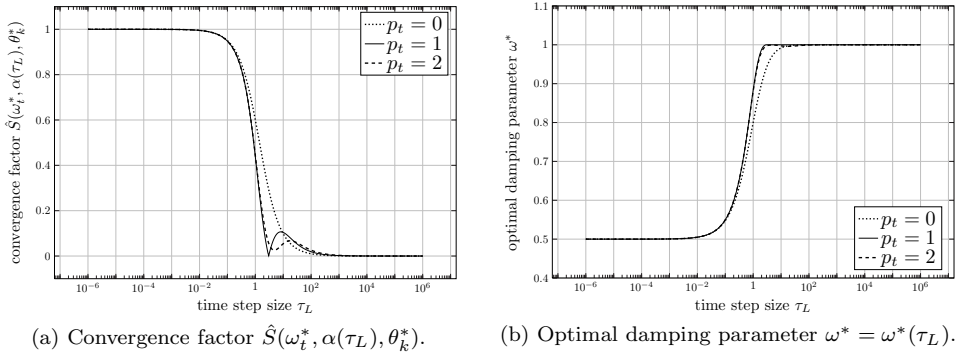

  \centering
  \subfloat[Convergence factor $\hat{S}(\omega_t^\ast, \alpha(\tau_L),
	\theta_k^\ast)$.]
  {
    \scalebox{0.43}{\input{./figures/plots/plotODEconvergence.tex}}
    \label{chap4_plot_convergence_smoother}
  }\hfill
  \subfloat[Optimal damping parameter $\omega^\ast = \omega^\ast(\tau_L)$.]
  {
      \scalebox{0.43}{\input{./figures/plots/plotODEomega.tex}}
      \label{chap4_plot1_best_damping}
  }
  \caption{Convergence factor $\hat{S}(\omega_t^\ast, \alpha(\tau_L),
	\theta_k^\ast)$ and optimal damping parameter  $\omega^\ast$}\label{chap4_plots_smoother}
\end{figure}

To show the convergence behavior of the damped block Jacobi smoother
(\ref{chap4_ODESmoother}) with respect to the time step size $\tau_L$, we 
now prove the following lemma for an arbitrary $\alpha \in \IR$.
\begin{lemma}\label{chap4_lemma8}
  For $\alpha \in \IR$ and the optimal choice for the damping parameter 
  \begin{align*}
    \omega_t^\ast = \begin{cases} \frac{1}{1+\alpha^2} & \alpha \geq 0,\\ 1
      &\alpha < 0, \end{cases}
  \end{align*}
  we have the estimate
  \begin{align*}
    \max_{\theta_k \in \Theta_L} \hat{S}(\omega_t^\ast,\alpha, \theta_k) \leq 
      \frac{\abs{\alpha}(1+\abs{\alpha})}{1+\alpha^2}.
  \end{align*}
\end{lemma}
\begin{proof}
  For the optimal damping parameter $\omega_t^\ast$, we have
  \[ \left(\hat{S}(\omega_t^\ast,\alpha, \theta_k)\right)^2 =
  \frac{\alpha^2}{1+\alpha^2} + \frac{2\alpha^3}{(1+\alpha^2)^2}
  \cos(\theta_k).\]
  For the case $\alpha \geq 0$ we therefore obtain
  \begin{align*}
    \argsup_{\theta_k \in [0,\pi]} \hat{S}(\omega_t^\ast,\alpha,
    \theta_k) = 0.
  \end{align*}
  Thus we have
  \[ \left( \hat{S}(\omega_t^\ast,\alpha, 0) \right)^2 = \frac{\alpha^2 +
    \alpha^4 + 2\alpha^3}{(1+\alpha^2)^2} =
  \frac{\alpha^2(1+\alpha)^2}{(1+\alpha^2)^2}. \]
  For the case  $\alpha < 0$ we find that 
   \begin{align*}
    \argsup_{\theta_k \in [0,\pi]} \hat{S}(\omega_t^\ast,\alpha,
    \theta_k) = \pi
  \end{align*}
  and thus
  \[ \left( \hat{S}(\omega_t^\ast,\alpha, 0) \right)^2 = \frac{\alpha^2 +
    \alpha^4 + 2\abs{\alpha}^3}{(1+\alpha^2)^2} =
  \frac{\alpha^2(1+\abs{\alpha})^2}{(1+\alpha^2)^2}. \]
  The statement of this lemma follows with the fact that
  \begin{align*}
    \max_{\theta_k \in \Theta_L} \hat{S}(\omega_t^\ast,\alpha, \theta_k) \leq
    \sup_{\theta_k \in [0,\pi]} \hat{S}(\omega_t^\ast,\alpha, \theta_k).
  \end{align*}
\end{proof}

\begin{remark}\label{chap4_remarkSmootherSolver}
  For sufficiently small values of $\alpha = \alpha(\tau_L)$, i.e. for
  sufficiently large time step sizes $\tau_L$, it is shown in Lemma
  \ref{chap4_lemma8} that the convergence factor
  $\hat{S}(\omega_t^\ast, \alpha(\tau_L), \theta_k^\ast)$ of the block
  Jacobi smoother (\ref{chap4_ODESmoother}) is close to zero, see
  Figure \ref{chap4_plot_convergence_smoother}. Hence the block Jacobi
  smoother (\ref{chap4_ODESmoother}) is already a very good iterative
  solver.
\end{remark}

All the estimates above are valid for arbitrary polynomial degrees
$p_t \in \IN_0$. For the limiting case $p_t \rightarrow \infty$, the
function $\alpha(\tau_L)$ is given by
\[ 
  \alpha(\tau_L) = e^{-\tau_L}, 
\]
since $\alpha(t)$ is the $(p_t,p_t+1)$ subdiagonal Pad\'{e}
approximation of the exponential function $e^{-t}$. Hence, the choice
of the best damping parameter $\omega^\ast$ and the smoothing factors
converge also to a limit function.

\subsection{Two-grid analysis}

We turn our attention now to the two-grid cycle for
(\ref{chap4_equationLinearSystemODE}), for which the error satisfies
\begin{align}\label{chap4_errorTwoGridCycle2}
  \vec{e}^{k+1} = \mathcal{M}_{\tau_L} \vec{e}^k := \mathcal{S}_{\tau_L}^{\nu_2}
  \left[ I - \mathcal{P}^L \mathcal{L}_{2 \tau_L}^{-1} \mathcal{R}^L
    \mathcal{L}_{\tau_L} \right] \mathcal{S}_{\tau_L}^{\nu_1} \vec{e}^k.
\end{align}
We use again Fourier mode analysis, which would be exact for time
periodic problems, see (\ref{chap4_periodicBC}). We thus need to
compute the Fourier symbol of the two-grid iteration matrix
$\mathcal{M}_{\tau_L}$. In Lemma \ref{chap4_lemma3} we already derived
the local Fourier symbol for the system matrix $\tilde{\mathcal{L}}_{\tau_L}$,
\[ 
  \hat{\mathcal{L}}_{\tau_L}(\theta_k) := K_{\tau_L} + M_{\tau_L} -
    e^{-\i \theta_k} N_{\tau_L} \in \IC^{N_t \times N_t},
\]
and the local Fourier symbol for the smoother
$\tilde{\mathcal{S}}_{\tau_L}^\nu$ is given by
\[ 
  \hat{\mathcal{S}}_{\tau_L}^\nu(\theta_k, \omega_t) := \left( (1-\omega_t)
      I_{N_t} + e^{-\i \theta_k} \omega_t (K_{\tau_L} + M_{\tau_L})^{-1}
      N_{\tau_L}\right)^\nu  \in \IC^{N_t \times N_t}, 
\]
see Lemma \ref{chap4_lemma3a}. For the local parts it is convenient
to use the so called stencil notation: for the
system matrix $\tilde{\mathcal{L}}_{\tau_L}$, its stencil is
\begin{align*}
 \widetilde{\mathcal{L}}_{\tau_L} := \begin{bmatrix} -N_{\tau_L} &
   K_{\tau_L}+M_{\tau_L} & 0  \end{bmatrix},
\end{align*}
and one smoothing iteration $\tilde{\mathcal{S}}_{\tau_L}^\nu$, $\nu = 1$, is given in
stencil notation by
\begin{align*}
 \widetilde{\mathcal{S}}_{\tau_L}^1 := \begin{bmatrix} -\omega_t (K_{\tau_L} +
   M_{\tau_L})^{-1}N_{\tau_L} & (1 - \omega_t)I_{N_t} & 0  \end{bmatrix}.
\end{align*}
 Using periodic boundary conditions leads to the mapping properties
\begin{align}\label{chap4_mappingProperties1}
 \tilde{\mathcal{L}}_{\tau_L} : \Psi_L(\theta_k) \rightarrow \Psi_L(\theta_k)
\quad \text{and} \quad \tilde{\mathcal{S}}_{\tau_L}^\nu : \Psi_L(\theta_k) \rightarrow
\Psi_L(\theta_k).
\end{align}
We next analyze the mapping properties of the restriction and
the prolongation operators, for which we need
\begin{lemma}\label{chap4_lemma9}
  The mapping $\gamma : \Theta_L^{\mathrm{low}} \rightarrow
  \Theta_L^{\mathrm{high}}$ with
  $\gamma(\theta_k) := \theta_k - \mathrm{sign}(\theta_k) \pi$
  is a one to one mapping.
\end{lemma}
\begin{proof}
  Let $\theta_k \in \Theta_L^{\mathrm{low}}$. By definition we have
  \[ \theta_k = \frac{2 k \pi}{N_L} \qquad \text{with} \quad k \in \left\{ 1-
    \frac{N_L}{4},\ldots,\frac{N_L}{4} \right\}. \]
  For the mapping $\gamma$ we then obtain
  \begin{align*}
    \gamma(\theta_k) =  \theta_k - \mathrm{sign}(\theta_k) \pi = \frac{2 k
      \pi}{N_L}  - \mathrm{sign}(\theta_k) \pi = \frac{2(k
      -\mathrm{sign}(\theta_k)\frac{N_L}{2} ) \pi}{N_L} = \frac{2 \hat{k}
      \pi}{N_L},
  \end{align*}
  with
  \begin{align*}
    \hat{k} = k - \mathrm{sign}(\theta_k)\frac{N_L}{2} \in \left\{ 1
      -\frac{N_L}{2},\ldots,-\frac{N_L}{4} \right\} \cup
    \left\{ \frac{N_L}{4} +1,\ldots, \frac{N_L}{2} \right\}.
  \end{align*}
  This implies that $\gamma(\theta_k) \in \Theta_L^{\mathrm{high}}$ and that
  $\mathrm{sign}(\gamma(\theta_k)) = -\mathrm{sign}(\theta_k)$. Hence we have
  \begin{align*}
    \gamma(\gamma(\theta_k)) = \gamma(\theta_k) -
    \mathrm{sign}(\gamma(\theta_k)) \pi = \gamma(\theta_k) +
    \mathrm{sign}(\theta_k) \pi = \theta_k,
  \end{align*}
 which completes the proof.
\end{proof}

\begin{lemma}\label{chap4_lemma10}
  The vector $\vec{u} = (\vec{u}_1, \vec{u}_2,\ldots,\vec{u}_{N_L})^\top \in
  \IR^{N_L\,N_t}$ for $N_{L-1}, N_t \in \IN$ and $N_L = 2 N_{L-1}$ can be written
  as
  \begin{align*}
    \vec{u} = \sum_{\theta_k \in \theta_L^{\mathrm{low}}} \left[
      \vec{\psi}^L(\theta_k) + \vec{\psi}^L(\gamma(\theta_k)) \right],
  \end{align*}
  where the vector $\vec{\psi}^L(\theta_k) \in \IC^{N_t N_L}$ is defined as in
  Lemma
  \ref{chap4_lemma1}.
\end{lemma}
\proof
{
  Applying Lemma \ref{chap4_lemma1} and Lemma \ref{chap4_lemma9} proves the
  statement of this lemma with
  \begin{align*}
     \vec{u} &= \sum_{\theta_k \in \Theta_L} \vec{\psi}^L(\theta_k) = \sum_{\theta_k
       \in \Theta_L^{\mathrm{low}}} \vec{\psi}^L(\theta_k) + \sum_{\theta_k \in
       \Theta_L^{\mathrm{high}}} \vec{\psi}^L(\theta_k)\\
      &= \sum_{\theta_k \in \Theta_L^{\mathrm{low}}} \vec{\psi}^L(\theta_k) +
      \sum_{\theta_k \in
       \Theta_L^{\mathrm{low}}} \vec{\psi}^L(\gamma(\theta_k)) = \sum_{\theta_k \in
       \Theta_L^{\mathrm{low}}} \left[ \vec{\psi}^L(\theta_k) +
       \vec{\psi}^L(\gamma(\theta_k)) \right].
  \end{align*}
}
Lemma \ref{chap4_lemma10} motivates
\begin{definition}[Space of harmonics]
   For $N_{L}, N_t \in \IN$ and for a low frequency $\theta_k\in
   \Theta_L^{\mathrm{low}}$ let the vector $\vec{\Phi}^L(\theta_k) \in \IC^{N_t
     N_L}$ be defined as in Lemma \ref{chap4_lemma1}. Then the linear space of
   harmonics with frequency $\theta_k$ is given by
  \begin{align*}
    \mathcal{E}_L(\theta_k) &:= \mathrm{span}\left\{ \vec{\Phi}^L(\theta_k),
      \vec{\Phi}^L(\gamma(\theta_k)) \right\} \\
    &\phantom{:}= \big\{ \vec{\psi}^L(\theta_k) \in \IC^{N_t N_L} :
    \vec{\psi}_n^L(\theta_k) =
      U_1 \vec{\Phi}_n^L(\theta_k) + U_2 \vec{\Phi}_n^L(\gamma(\theta_k)), \\
      &\qquad\qquad\qquad\qquad\qquad\qquad n = 1,\ldots,N_L \text{ and } U_1,
      U_2 \in
      \IC^{N_t \times N_t} \big\}.
\end{align*}
\end{definition}
Under the assumption of periodic boundary conditions, the mappings
(\ref{chap4_mappingProperties1}) imply the mapping properties
\begin{align}\label{chap4_mappingProperties2}
  \tilde{\mathcal{L}}_{\tau_L} : \mathcal{E}_L(\theta_k) \rightarrow
  \mathcal{E}_L(\theta_k)
\quad \text{and} \quad \tilde{\mathcal{S}}_{\tau_L}^\nu : \mathcal{E}_L(\theta_k)
\rightarrow \mathcal{E}_L(\theta_k),
\end{align}
with the mapping for the system matrix $\tilde{\mathcal{L}}_{\tau_L}$
\begin{align}
    &\begin{pmatrix} U_1\\U_2 \end{pmatrix} \mapsto \begin{pmatrix}
    \hat{\mathcal{L}}_{\tau_L}(\theta_k) & 0 \\ 0 &
    \hat{\mathcal{L}}_{\tau_L}(\gamma(\theta_k)) \end{pmatrix} \begin{pmatrix}
    U_1\\U_2 \end{pmatrix},\\
    \intertext{and the mapping for the smoother
      $\tilde{\mathcal{S}}_{\tau_L}^\nu$}
    &\begin{pmatrix} U_1\\U_2 \end{pmatrix} \mapsto \begin{pmatrix}
    \hat{\mathcal{S}}_{\tau_L}^\nu(\theta_k, \omega_t) & 0 \\ 0 &
    \hat{\mathcal{S}}_{\tau_L}^\nu(\gamma(\theta_k),
    \omega_t) \end{pmatrix} \begin{pmatrix} U_1\\U_2 \end{pmatrix}.
\end{align}
We now analyze the two-grid cycle on the space of harmonics
$\mathcal{E}_L(\theta_k)$ for frequencies $\theta_k \in
\Theta_L^{\mathrm{low}}$. To do so, we further have to investigate the mapping
properties of the restriction and prolongation operators $\mathcal{R}^L$ and
$\mathcal{P}^L$. The restriction operator is given by
\begin{align}
  \mathcal{R}^L &:= \begin{pmatrix}
    R_1 & R_2 &     &        &        & \\
        &     & R_1 & R_2    &        & \\
        &     &     & \ddots & \ddots & \\
        &     &     &        & R_1    & R_2
  \end{pmatrix} \in \IR^{N_t N_L \times N_t N_{L-1}},\label{chap4_restrictionTime}\\
  \intertext{and the prolongation operator is given by}
  \mathcal{P}^L &:= \begin{pmatrix}
    R_1^\top &          &        &         \\
    R_2^\top &          &        &         \\
            & R_1^\top  &        &         \\
            & R_2^\top  & \ddots &         \\
            &          & \ddots & R_1^\top \\
            &          &        & R_2^\top 
  \end{pmatrix} = (\mathcal{R}^L)^\top \in \IR^{N_t N_{L-1} \times N_t
    N_L},\label{chap4_prolongationTime}
\end{align}
with the local prolongation matrices
\begin{align*}
  R_1^\top := M_{\tau_L}^{-1} \widetilde{M}_{\tau_L}^1 \quad \text{and} \quad
  R_2^\top := M_{\tau_L}^{-1} \widetilde{M}_{\tau_L}^2,
\end{align*}
where for basis functions $\left\{ \psi_k \right\}_{k=1}^{N_t} \subset
\IP^{p_t}(0,\tau_L)$ and $\left\{  \widetilde{\psi}_k \right\}_{k=1}^{N_t}
\subset \IP^{p_t}(0, 2\tau_L)$ the local projection matrices from coarse to fine
grids are defined for $k,\ell = 1,\ldots,N_t$ by
\begin{align*}
  \widetilde{M}_{\tau_L}^1[k,\ell] := \int_{0}^{\tau_L}
  \widetilde{\psi}_\ell(t) \psi_k(t)
  \mathrm{d}t \quad \text{and} \quad \widetilde{M}_{\tau_L}^2[k,\ell] :=
  \int_{\tau_L}^{2 \tau_L} \widetilde{\psi}_\ell(t) \psi_k(t+\tau) \mathrm{d}t.
\end{align*}
To prove the mapping properties of the restriction operator
$\mathcal{R}^L$ we need
\begin{lemma}\label{chap4_lemma11}
  Let $\vec{\psi}^L(\theta_k) \in \Psi_L(\theta_k)$ for $\theta_k \in \Theta_L$.
  Then $\psi_{2n}^L(\theta_k) = \psi_{n}^L(2 \theta_k)$ holds
  for $n = 1,\ldots,N_{L-1}$.
\end{lemma}
\begin{proof}
   Let $\vec{\psi}^L(\theta_k) \in \Psi_L(\theta_k)$. Hence we have
   $\vec{\psi}_n^L(\theta_k) = U \vec{\Phi}_{n}^L(\theta_k)$ for $n =
   1,\dots,N_L$. Then for $\vec{\Phi}_{2n}^L(\theta_k)$ with $n \in
   \{1,\ldots,N_{L-1} \}$ we obtain for $\ell = 1,\ldots,N_t$ that
   \begin{align*}
     \vec{\Phi}_{2n}^L(\theta_k)[\ell] = \vec{\varphi}_{2n}(\theta_k) = e^{\i
       2n \theta_k} = \vec{\varphi}_{n}(2 \theta_k) =
     \vec{\Phi}_{n}^L(2\theta_k)[\ell].
   \end{align*}
   Hence the statement of this lemma follows from
   \[ \vec{\psi}_{2n}^L(\theta_k) = U \vec{\Phi}_{2n}^L(\theta_k) = U
   \vec{\Phi}_{n}^L(2\theta_k) = \vec{\psi}_{n}^L(2\theta_k). \]
\end{proof}

The next two lemmas give the mapping properties of the restriction and
extension:
\begin{lemma}\label{chap4_lemma12}
  Let $\theta_k \in \Theta_L^{\mathrm{low}}$. Then the restriction
  operator $\mathcal{R}^L$ has the mapping property
  \begin{align*}
    \mathcal{R}^L : \mathcal{E}_L(\theta_k) \rightarrow \Psi_{L-1}(2 \theta_k),
  \end{align*}
  with the mapping
  \begin{align*}
    \begin{pmatrix} U_1\\U_2 \end{pmatrix} \mapsto \begin{pmatrix}
    \hat{\mathcal{R}}(\theta_k) &
    \hat{\mathcal{R}}(\gamma(\theta_k)) \end{pmatrix} \begin{pmatrix}
    U_1\\U_2 \end{pmatrix} \in \IC^{N_t \times N_t}
  \end{align*}
  and the Fourier symbol 
  \begin{align*}
    \hat{\mathcal{R}}(\theta_k) := e^{-\i \theta_k} R_1 + R_2.
  \end{align*}
\end{lemma}
\begin{proof}
  Let $\vec{\psi}^L(\theta_k) \in \mathcal{E}_L(\theta_k)$ for some frequency
  $\theta_k \in \Theta_L^{\mathrm{low}}$ with the linear combination
  $\vec{\psi}_n^L(\theta_k) = U_1 \vec{\Phi}_n^L(\theta_k) + U_2
  \vec{\Phi}_n^L(\gamma(\theta_k))$. Then for the Fourier mode
  $\vec{\Phi}^L(\theta_\ell)$ with frequency $\theta_\ell \in \Theta_L$ we
  have for a fixed $n \in \{ 1, \ldots,N_{L-1}  \}$
  \begin{align*}
    \left( \mathcal{R}^L \vec{\Phi}^L(\theta_\ell)\right)_n &= R_1
    \vec{\Phi}_{2n-1}^L(\theta_\ell) + R_2 \vec{\Phi}_{2n}^L(\theta_\ell)
    = \left[ e^{-\i \theta_\ell} R_1 + R_2 \right]
    \vec{\Phi}_{2n}^L(\theta_\ell)\\
    &= \left[ e^{-\i \theta_\ell} R_1 + R_2 \right]
    \vec{\Phi}_{n}^{L-1}(2\theta_\ell),
  \end{align*}
  since $\vec{\Phi}^L(\theta_\ell) \in \Psi_L(\theta_\ell)$ using
  Lemma \ref{chap4_lemma2}, and also applying Lemma
  \ref{chap4_lemma11}. Using this result for the vector
  $\vec{\psi}^L(\theta_k)$ leads to
  \begin{align*}
    \left( \mathcal{R}^L \vec{\psi}^L(\theta_k) \right)_n =
    \hat{\mathcal{R}}(\theta_k) U_1 \vec{\Phi}_{n}^{L-1}(2\theta_k) +
    \hat{\mathcal{R}}(\gamma(\theta_k))
    U_2\vec{\Phi}_{n}^{L-1}(2\gamma(\theta_k)).
  \end{align*}
  For $i = 1, \ldots, N_t$ we further have that
  \begin{align*}
    \vec{\Phi}_n^{L-1}(2\gamma(\theta_k))[i] &=
    \vec{\varphi}_n(2\gamma(\theta_k)) =
    e^{\i n 2 \gamma(\theta_k)} = e^{\i n 2 \theta_k - \i
      \,\mathrm{sign}(\theta_k) 2 \pi} \\
    &=  e^{\i n 2 \theta_k} = \vec{\varphi}_n(2\theta_k) = \vec{\Phi}_n^{L-1}(2
    \theta_k)[i].
  \end{align*}
  Hence we obtain
  \begin{align*}
    \left( \mathcal{R}^L \vec{\psi}^L(\theta_k) \right)_n = \left[
      \hat{\mathcal{R}}(\theta_k) U_1 + \hat{\mathcal{R}}(\gamma(\theta_k)) U_2
    \right] \vec{\Phi}_{n}^{L-1}(2\theta_k),
  \end{align*}
  which completes the proof.
\end{proof}
\begin{lemma}\label{chap4_lemma13}
  Let $\theta_k \in \Theta_L^{\mathrm{low}}$. Then the the
  prolongation operator $\mathcal{P}^L$ has the mapping property 
  \begin{align*}
    \mathcal{P}^L : \Psi_{L-1}(2 \theta_k) \rightarrow \mathcal{E}_L(\theta_k),
  \end{align*}
  with the mapping
  \begin{align*}
    U \mapsto \begin{pmatrix}
    \hat{\mathcal{P}}(\theta_k) \\
    \hat{\mathcal{P}}(\gamma(\theta_k)) \end{pmatrix} U \in \IC^{2 N_t \times N_t}
  \end{align*}
  and the Fourier symbol 
  \begin{align*}
    \hat{\mathcal{P}}(\theta_k) := \frac{1}{2}\left[e^{\i \theta_k} R_1^\top +
      R_2^\top\right].
  \end{align*}
\end{lemma}
\begin{proof}
  For $\theta_k \in \Theta_{L}^{\mathrm{low}}$ let $\vec{\psi}^{L-1}(2\theta_k) \in
  \Psi_{L-1}(2\theta_k)$ with $\vec{\psi}_{\hat{n}}^{L-1}(2\theta_k) = U
  \vec{\Phi}_{\hat{n}}^{L-1}(2\theta_k)$ for $\hat{n} \in \left\{ 1,\ldots,N_{L-1}
  \right\}$. We then define $\vec{\psi}^{L}(\theta_k) \in
  \Psi_{L}(\theta_k)$ as $\vec{\psi}_{n}^{L}(\theta_k) = U
  \vec{\Phi}_{n}^{L}(\theta_k)$ for $n \in \left\{
    1,\ldots,N_{L}\right\}$ and obtain
  \begin{align*}
    \left( \mathcal{P}^L \vec{\psi}^{L-1}(2\theta_k) \right)_{2 \hat{n}-1} =
    R_1^\top
    \vec{\psi}_{\hat{n}}^{L-1}(2\theta_k)
    = R_1^\top \vec{\psi}_{2\hat{n}}^{L}(\theta_k)
    = e^{\i \theta_k} R_1^\top \vec{\psi}_{2\hat{n}-1}^{L}(\theta_k),
  \end{align*}
  where we used Lemma \ref{chap4_lemma11} and Lemma
  \ref{chap4_lemma2}.  Similar computations also give
  \begin{align*}
    \left( \mathcal{P}^L \vec{\psi}^{L-1}(2\theta_k) \right)_{2 \hat{n}} &= R_2^\top
    \vec{\psi}_{\hat{n}}^{L-1}(2\theta_k) =
    R_2^\top\vec{\psi}_{2\hat{n}}^{L}(\theta_k).
  \end{align*}
  Hence we have for $n \in \left\{ 1,\ldots,N_L \right\}$
  \begin{align*}
    \left( \mathcal{P}^L \vec{\psi}^{L-1}(2\theta_k) \right)_n = \begin{cases}
      e^{\i
        \theta_k} R_1^\top \vec{\psi}_n^L(\theta_k) & n \text{ odd},\\ R_2^\top
      \vec{\psi}_n^L(\theta_k) & n \text{ even} \end{cases} \in \IC^{N_t}.
  \end{align*}
  For the image of the prolongation operator $\mathcal{P}^L$ to be
  contained in $\mathcal{E}_L(\theta_k)$, the
  following equations have to be satisfied for $n = 1,\ldots,N_L$:
  \begin{equation}\label{chap4_proofProlongationEqu1}
  \begin{aligned}
    U_1 \vec{\Phi}_n^L(\theta_k) + U_2 \vec{\Phi}_n^L(\gamma(\theta_k)) &= e^{\i
        \theta_k} R_1^\top U \vec{\Phi}_n^L(\theta_k) &\quad &\text{for } n
      \text{ odd},\\
    U_1 \vec{\Phi}_n^L(\theta_k) + U_2 \vec{\Phi}_n^L(\gamma(\theta_k)) &=
    R_2^\top U \vec{\Phi}_n^L(\theta_k) &\quad &\text{for } n
      \text{ even}.
  \end{aligned}
  \end{equation}
  Further computations show for $\ell = 1,\ldots,N_t$ that
  \begin{align*}
    \vec{\Phi}_n^L(\gamma(\theta_k))[\ell] &= \vec{\varphi}_n(\gamma(\theta_k))
    = e^{\i n \gamma(\theta_k)} = e^{\i n \theta_k - \i \,
      \mathrm{sign}(\theta_k) n \pi} = \vec{\varphi}_n(\theta_k) e^{\i
      \,\mathrm{sign}(\theta_k) n \pi}\\
    &= \vec{\Phi}_n^L(\theta_k)[\ell] \begin{cases} 1 & n \text{ even}, \\ -1
      & n \text{ odd}. \end{cases}
  \end{align*}
  Hence the equations (\ref{chap4_proofProlongationEqu1}) are equivalent to the
  system of linear equations
  \begin{align*}
    U_1 - U_2 &= e^{\i \theta_k} R_1^\top U,\\
    U_1 + U_2 &= R_2^\top U.
  \end{align*}
  Solving for $U_1$ and $U_2$ results in
  \begin{align*}
    U_1 &= \frac{1}{2} \left[ e^{\i \theta_k} R_1^\top + R_2^\top \right] =
    \hat{\mathcal{P}}(\theta_k) U,\\
    U_2 &= \frac{1}{2} \left[ -e^{\i \theta_k} R_1^\top + R_2^\top \right] =
    \frac{1}{2} \left[ e^{\i (\theta_k - \mathrm{sign}(\theta_k) \pi)} R_1^\top
      + R_2^\top \right]\\
    &= \frac{1}{2} \left[ e^{\i \gamma(\theta_k)} R_1^\top + R_2^\top \right] =
    \hat{\mathcal{P}}(\gamma(\theta_k)) U,
  \end{align*}
  which completes the proof.
\end{proof}

In view of Lemma \ref{chap4_lemma12} and Lemma \ref{chap4_lemma13},
the stencil notations for the restriction and prolongation operators
$\mathcal{R}^L$ and $\mathcal{P}^L$ are given by
\begin{align*}
  \widetilde{\mathcal{R}}^L :=  \begin{bmatrix} R_1 & R_2 &
    0  \end{bmatrix} \quad \text{and} \quad \widetilde{\mathcal{P}}^L :=
  \frac{1}{2} \begin{bmatrix} 0 & R_2^\top &  R_1^\top  \end{bmatrix}.
\end{align*}
For the two-grid operator $\mathcal{M}_{\tau_L}$ it now remains to
prove the mapping property of the coarse grid operator
$\tilde{\mathcal{L}}_{2\tau_L}^{-1}$. Assuming periodic
boundary conditions, we have for $\theta_k \in \Theta_{L-1}$ by using
(\ref{chap4_mappingProperties1}) that
\begin{align*}
 \tilde{\mathcal{L}}_{2\tau_L}^{-1} : \Psi_{L-1}(\theta_k) \rightarrow
 \Psi_{L-1}(\theta_k),
\end{align*}
with the Fourier symbol
\begin{align*}
  \hat{\mathcal{L}}_{2\tau_L}^{-1}(\theta_k) = \left(K_{\tau_L} + M_{\tau_L} -
e^{-\i \theta_k} N_{\tau_L} \right)^{-1} =
(\hat{\mathcal{L}}_{2\tau_L}(\theta_k))^{-1} \in \IC^{N_t \times N_t}.
\end{align*}
\begin{lemma}\label{chap4_lemma14}
  The frequency mapping 
  \[ \beta: \Theta_L^{\mathrm{low}} \rightarrow \Theta_{L-1}\qquad \text{with}
  \qquad \theta_k \mapsto 2 \theta_k \]
  is a one to one mapping.
\end{lemma}
\begin{proof}
  For $\theta_k \in \Theta_L^{\mathrm{low}}$ we obtain
  \begin{align*}
    \beta(\theta_k) = 2 \theta_k = 2 \frac{2 k \pi}{N_L} = \frac{2 k
      \pi}{\frac{N_L}{2}} = \frac{2 k \pi}{N_{L-1}} \in \Theta_{L-1}.
  \end{align*}
  The proof of this lemma then follows from the identity
  \begin{align*}
    k \in \left\{ 1 - \frac{N_L}{4},\ldots,\frac{N_L}{4} \right\} = \left\{ 1 -
      \frac{N_{L-1}}{2},\ldots,\frac{N_{L-1}}{2} \right\}.
  \end{align*}
\end{proof}

With Lemma \ref{chap4_lemma14} we now have for $\theta_k \in
\Theta_L^{\mathrm{low}}$ the coarse grid operator mapping property
\begin{align}\label{chap4_mappingProperties3}
  \tilde{\mathcal{L}}_{2\tau_L}^{-1} : \Psi_{L-1}(2 \theta_k) \rightarrow
 \Psi_{L-1}(2\theta_k).
\end{align}
We are now able to prove the following theorem for the two-grid operator
$\mathcal{M}_{\tau_L}$.
\begin{theorem}\label{chap4_theorem4}
   Let $\theta_k \in \Theta_L^{\mathrm{low}}$. With time periodic
   boundary conditions, the two-grid operator $\mathcal{M}_{\tau_L}$
   has the mapping property
  \begin{align*}
    \mathcal{M}_{\tau_L} : \mathcal{E}_L(\theta_k) \rightarrow
    \mathcal{E}_L(\theta_k),
  \end{align*}
  with the mapping
  \begin{align*}
    \begin{pmatrix} U_1\\U_2 \end{pmatrix} \mapsto 
    \hat{\mathcal{M}}(\theta_k) \begin{pmatrix}
    U_1\\U_2 \end{pmatrix}
  \end{align*}
  and the iteration matrix
  \begin{align*}
    \hat{\mathcal{M}}(\theta_k) := \begin{pmatrix}
    \hat{\mathcal{S}}_{\tau_L}^{\nu_2}(\theta_k, \omega_t) & 0 \\ 0 &
    \hat{\mathcal{S}}_{\tau_L}^{\nu_2}(\gamma(\theta_k),
    \omega_t) \end{pmatrix}
 \mathcal{K}(\theta_k)
  \begin{pmatrix}
    \hat{\mathcal{S}}_{\tau_L}^{\nu_1}(\theta_k, \omega_t) & 0 \\ 0 &
    \hat{\mathcal{S}}_{\tau_L}^{\nu_1}(\gamma(\theta_k),
    \omega_t) \end{pmatrix}
  \end{align*}
  with
  \begin{align*}
    \mathcal{K}(\theta_k) := I_{2 N_t} - 
    \begin{pmatrix}
    \hat{\mathcal{P}}(\theta_k) \\
    \hat{\mathcal{P}}(\gamma(\theta_k)) \end{pmatrix}
    (\hat{\mathcal{L}}_{2\tau_L}(2\theta_k))^{-1}
    \begin{pmatrix}\hat{\mathcal{R}}(\theta_k)^\top \\
    \hat{\mathcal{R}}(\gamma(\theta_k))^\top \end{pmatrix}^\top
    \begin{pmatrix}
    \hat{\mathcal{L}}_{\tau_L}(\theta_k) & 0 \\ 0 &
    \hat{\mathcal{L}}_{\tau_L}(\gamma(\theta_k)) \end{pmatrix}.
 \end{align*}
\end{theorem}
\begin{proof}
  The statement of this theorem is a direct consequence of Lemma
  \ref{chap4_lemma12}, Lemma \ref{chap4_lemma13} and the mapping properties
  (\ref{chap4_mappingProperties2}) and (\ref{chap4_mappingProperties3}).
\end{proof}

We now write the initial error $\vec{e}^0 = \vec{u}-\vec{u}^0$ as
\[ \vec{e}^0 = \sum_{\theta_k \in \Theta_L^{\mathrm{low}}} \left[
  \vec{\psi}^L(\theta_k) + \vec{\psi}^L(\gamma(\theta_k)) \right], \]
with $\vec{\psi}^L(\theta_k) + \vec{\psi}^L(\gamma(\theta_k)) \in
\mathcal{E}_L(\theta_k)$ for all $\theta_k \in
\Theta_L^{\mathrm{low}}$, see Lemma \ref{chap4_lemma10}. In view of
Theorem \ref{chap4_theorem4} we can analyze the asymptotic
behavior of the two-grid cycle by simply computing the
largest spectral radius of $\hat{\mathcal{M}}(\theta_k) \in \IC^{2 N_t
  \times 2 N_t}$ with respect to the low frequencies $\theta_k \in
\Theta_L^{\mathrm{low}}$. This motivates
\begin{definition}[Asymptotic two-grid convergence factor]
  For the two-grid iteration matrix $\mathcal{M}_{\tau_L}$, we define the
  asymptotic convergence factor 
  \begin{align*}
    \varrho(\mathcal{M}_{\tau_L}) := \max\left\{
      \varrho\left(\hat{\mathcal{M}}(\theta_k)\right) : \theta_k \in
      \Theta_L^{\mathrm{low}} \right\}.
  \end{align*}
\end{definition}

For the simplest case, i.e. for the polynomial degree $p_t = 0$, we have to
compute the spectral radius of the $2\times 2$ iteration matrix
$\hat{\mathcal{M}}(\theta_k)$. Using one pre and post smoothing step,
i.e. $\nu_1 = \nu_2 = 1$, we find that the spectral radius of
$\hat{\mathcal{M}}(\theta_k) \in \IC^{2 \times 2}$ is
\begin{align*}
\varrho\left(\hat{\mathcal{M}}(\theta_k)\right) = \abs{\frac{4 (1 + \tau_L)^2
    \left(\sin(\theta_k)\right)^2 + \tau_L^2 (1 + 2\tau_L - e^{2 \i
      \theta_k})}{(2 + \tau_L (2 + \tau_L))^2 \left( (1 + 2\tau_L)e^{2 \i
        \theta_k} -1 \right)}}.
\end{align*}
Further calculations show that the maximum of
$\varrho\left(\hat{\mathcal{M}}(\theta_k)\right)$ with respect to the low
frequencies $\theta_k \in \Theta_L^{\mathrm{low}}$  is obtained for
$\theta_k^\ast = \frac{\pi}{2}$. Hence for this simple case we can compute the
asymptotic convergence factor explicitly,
\begin{align*}
   \varrho(\mathcal{M}_{\tau_L}) = \frac{1}{2+2\tau_l + \tau_L^2} \in
   [0,\frac{1}{2}]
   \qquad \text{for all } \tau_L \geq 0.
\end{align*}
For periodic boundary conditions we therefore conclude that the two-grid cycle
converges for any  $\tau_L \geq 0$ to the exact solution, since
$\varrho(\mathcal{M}_{\tau_L}) \leq \frac{1}{2}$ for all $\tau_L \geq
0$. Furthermore, we obtain that the asymptotic convergence factor
$\varrho(\mathcal{M}_{\tau_L})$ gets very small for large time step sizes,
i.e. $\varrho(\mathcal{M}_{\tau_L}) = \mathcal{O}(\tau_L^{-2})$. This results
from the fact that the smoother itself is already an efficient iterative
solver for large time step sizes, see Remark \ref{chap4_remarkSmootherSolver}.

For higher polynomial degrees $p_t$, we have to compute the eigenvalues
of the $2(p_t+1)\times 2(p_t+1)$ iteration matrix
$\hat{\mathcal{M}}(\theta_k)$, which are difficult to obtain in closed
form. We thus compute numerically for all frequencies
$\theta_k \in \Theta_L^{\mathrm{low}}$ the eigenvalues of
$\hat{\mathcal{M}}(\theta_k)$ to determine the asymptotic
convergence factor $\varrho\left(\hat{\mathcal{M}}(\theta_k)\right)$
for a given time step size $\tau_L$.

We show in Figures \ref{chap4PlotODEP0}--\ref{chap4PlotODEP5} the
theoretical asymptotic convergence factors
$\varrho\left(\hat{\mathcal{M}}(\theta_k)\right)$ as solid lines for
$\tau_L \in [10^{-6},10^6]$ and $p_t \in \{0,1,\ldots,5\}$ and three
numbers of smoothing iterations $\nu_1 = \nu_2 = \nu$ with $\nu \in
\{1,2,5\}$. We see that for higher polynomial degrees $p_t \geq 1$ the
theoretical convergence factors are about half the theoretical
convergence factor of the lowest order case $p_t = 0$. We also notice
that the theoretical convergence factors are close to zero for large
time step sizes $\tau_L$, as expected, see Remark
\ref{chap4_remarkSmootherSolver}. Furthermore, for odd polynomial
degrees $p_t$ we observe a peak in the plots for the theoretical
convergence factors. This is because
for odd polynomial degrees, the $(p_t,p_t+1)$ subdiagonal Pad\'{e}
approximation of $e^{-t}$ has exactly one zero for $t > 0$. Hence for
one $\tau_L^\ast > 0$ we have $\alpha(\tau_L^\ast) = 0$ which implies
for the smoothing factor $\mu_S=0$, see Lemma
\ref{chap4_lemma8}. Hence the application of only two smoothing
iterations results in an exact solver.

We also show in the same Figures
\ref{chap4PlotODEP0}--\ref{chap4PlotODEP5}, using dots, triangles and
squares, the numerically computed convergence factors when solving the
equation
\[ \mathcal{L}_{\tau_L} \vec{u} = \vec{f} \]
with our two-grid cycle. We use $N_L = 1024$ time steps with a zero
right hand side, i.e. $\vec{f} = \mathbf{0}$, and a random initial
vector $\vec{u}^0$ with values between zero and one. The numerical 
convergence factor we measure is
\begin{align*}
  \max_{k = 1,\ldots,N_{\mathrm{iter}}}
  \frac{\norm{\vec{r}^{k+1}}_2}{\norm{\vec{r}^{k}}_2}, \quad \text{with
  }\vec{r}^k := \vec{f} - \mathcal{L}_{\tau_L} \vec{u}^k,
\end{align*}
where $N_{\mathrm{iter}} \in \IN$, $N_{\mathrm{iter}} \leq 250$ is the
number of two-grid iterations used until we have reached a given
relative error reduction of $\varepsilon_{\mathrm{MG}}$. To measure
the asymptotic behavior of the two-grid cycle, we have to use quite a
small tolerance $\varepsilon_{\mathrm{MG}} = 10^{-140}$, since in the
pre-asymptotic range the convergence rates of the two-grid cycle are in
fact even better than our asymptotic estimate.
We see that the theoretical results from the Fourier
mode analysis agree very well with the numerical results, even though the
Fourier mode analysis is only rigorous for time periodic
conditions.

 \begin{figure}
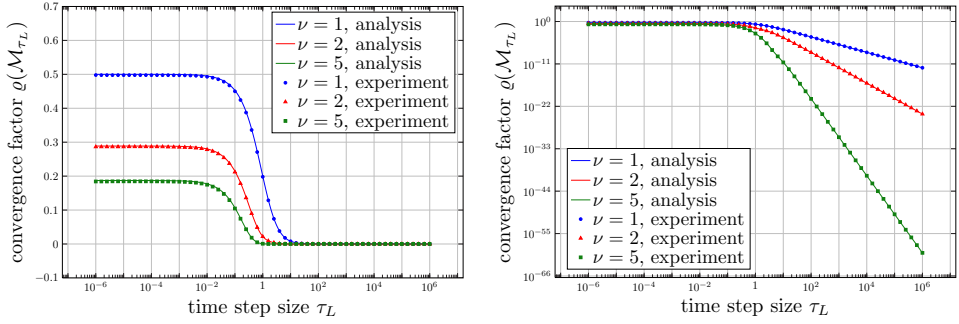

    \centering
    \subfloat
    {
      \scalebox{0.429}{\input{./figures/plots/plotODE0b.tex}}
    }\hfill
    \subfloat
    {
       \scalebox{0.429}{\input{./figures/plots/plotODElog0.tex}}
    }
    \caption{Average convergence factor
      $\varrho\left(\mathcal{M}_{\tau_L}\right)$ for different time step
      sizes $\tau_L$, $p_t = 0$ and numerical convergence rates for $N_t = 1024$
      time steps. Log-linear plot (top) and Log-log plot
      (bottom).}\label{chap4PlotODEP0}
  \end{figure}
  \begin{figure}%[htpb]
    \centering
    \subfloat
    {
      \scalebox{0.429}{\input{./figures/plots/plotODE1.tex}}
    }\hfill
    \subfloat
    {
      \scalebox{0.429}{\input{./figures/plots/plotODElog1.tex}}
    }
    \caption{Average convergence factor
      $\varrho\left(\mathcal{M}_{\tau_L}\right)$ for different time step
      sizes $\tau_L$, $p_t = 1$ and numerical convergence rates for $N_t = 1024$
      time steps. Log-linear plot (top) and Log-log plot
      (bottom). }\label{chap4PlotODEP1}
  \end{figure}
  \begin{figure}%[htpb]
    \centering
    \subfloat
    {
      \scalebox{0.429}{\input{./figures/plots/plotODE2.tex}}
    }\hfill
    \subfloat
    {
      \scalebox{0.429}{\input{./figures/plots/plotODElog2.tex}}
    }
    \caption{Average convergence factor
      $\varrho\left(\mathcal{M}_{\tau_L}\right)$ for different time step
      sizes $\tau_L$, $p_t = 2$ and numerical convergence rates for $N_t = 1024$
      time steps. Log-linear plot (top) and Log-log plot
      (bottom). }\label{chap4PlotODEP2}
  \end{figure}
  \begin{figure}%[htpb]
    \centering
    \subfloat
    {
      \scalebox{0.429}{\input{./figures/plots/plotODE3.tex}}
    }\hfill
    \subfloat
    {
      \scalebox{0.429}{\input{./figures/plots/plotODElog3.tex}}
    }
    \caption{Average convergence factor
      $\varrho\left(\mathcal{M}_{\tau_L}\right)$ for different time step
      sizes $\tau_L$, $p_t = 3$ and numerical convergence rates for $N_t = 1024$
      time steps. Log-linear plot (top) and Log-log plot
      (bottom). }\label{chap4PlotODEP3}
  \end{figure}
  \begin{figure}%[htpb]
    \centering
    \subfloat
    {
      \scalebox{0.429}{\input{./figures/plots/plotODE4.tex}}
    }\hfill
    \subfloat
    {
      \scalebox{0.429}{\input{./figures/plots/plotODElog4.tex}}
    }
    \caption{Average convergence factor
      $\varrho\left(\mathcal{M}_{\tau_L}\right)$ for different time step
      sizes $\tau_L$, $p_t = 4$ and numerical convergence rates for $N_t = 1024$
      time steps. Log-linear plot (top) and Log-log plot
      (bottom). }\label{chap4PlotODEP4}
  \end{figure}
  \begin{figure}
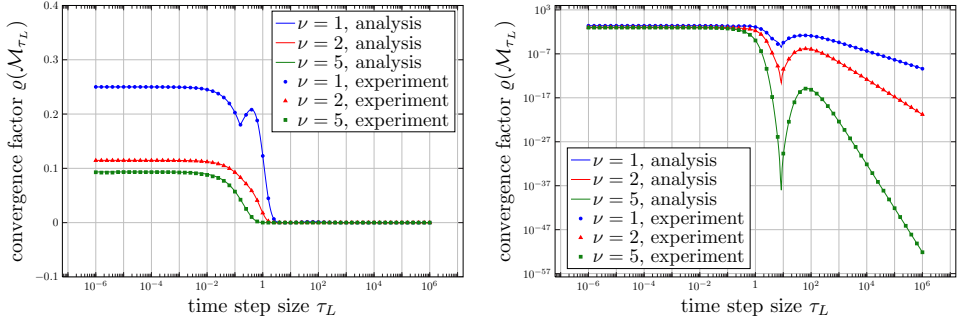
%[htpb]
    \centering
    \subfloat
    {
      \scalebox{0.429}{\input{./figures/plots/plotODE5.tex}}
    }\hfill
    \subfloat
    {
      \scalebox{0.429}{\input{./figures/plots/plotODElog5.tex}}
    }
    \caption{Average convergence factor
      $\varrho\left(\mathcal{M}_{\tau_L}\right)$ for different time step
      sizes $\tau_L$, $p_t = 5$ and numerical convergence rates for $N_t = 1024$
      time steps. Log-linear plot (top) and Log-log plot
      (bottom). }\label{chap4PlotODEP5}
  \end{figure}

%%%%%%%%%%%%%%%%%%%%%%%%%%%%%%%%%%%%%%%%%%%%%%%%%%%%%%%%%%%%%%%%%%%%%%%%%%%%%%%%
%%%%%%%%%%%%%%%%%%%%%%%%%%%%%%%%%%%%%%%%%%%%%%%%%%%%%%%%%%%%%%%%%%%%%%%%%%%%%%%%
\section{Numerical example}\label{NumericalSec}
%%%%%%%%%%%%%%%%%%%%%%%%%%%%%%%%%%%%%%%%%%%%%%%%%%%%%%%%%%%%%%%%%%%%%%%%%%%%%%%%
%%%%%%%%%%%%%%%%%%%%%%%%%%%%%%%%%%%%%%%%%%%%%%%%%%%%%%%%%%%%%%%%%%%%%%%%%%%%%%%%

In this example we test the weak and strong scaling behavior of our
new time multigrid algorithm. We use different polynomial
degrees $p_t \in \{ 0,1,5,10,20 \}$ and a fixed time step size $\tau =
10^{-6}$.  For a random initial guess and a zero right hand side we
run the algorithm until we have reached a relative error reduction of
$\varepsilon_{\mathrm{MG}} = 10^{-8}$.  We first study the
weak scaling behavior by using a fixed number of time
steps per core ($32\; 768$), and we increase the number
of cores when increasing the number of time steps.  In
Table \ref{chap4_TableWeak}, we give computation times for different
numbers of cores and  polynomial degrees. We observe
excellent weak scaling, i.e. the computation times remain
bounded when we increase the number of cores.
We next study the strong scaling behavior by fixing the problem size,
i.e. we use $1\;048\;576$ time steps in this example. Then we increase
the number of cores from $1$ up to $32\;768$. In Table
\ref{chap4_TableStrong} the computation times are given for different
number of cores and polynomial degrees. We observe that the
computation costs are basically divided by a factor of
two if we double the number of cores, only for $32\;768$ cores and
$p_t \in \{0,1,5\}$ we obtain no speedup any more, since the local
problems are to small, i.e. for $p_t = 0$ one core has to solve for
only $32$ unknowns.

These computations were performed on the Monte Rosa supercomputer at 
the Swiss National Supercomputing Centre CSCS in Lugano.

\begin{table}
  \centering
    \subfloat[Weak scaling results.]
    {
      \scalebox{0.9}{
      \begin{tabular}{r|r|r|r|r|r|r}
        \thickhline
        cores & time steps & $p_t=0$ & $p_t =1$ & $p_t=5$ & $p_t=10$ & $p_t=20$ \\ \hline
        $1$ & $32\;768$ & $3.96$ & $2.45$ & $3.60$ & $7.97$ & $17.77$ \\
        $2$ & $65\;536$ & $5.30$ & $3.39$ & $5.10$ & $10.74$ & $25.83$ \\
        $4$ & $131\;072$ & $5.34$ & $3.38$ & $5.45$ & $10.84$ & $26.13$ \\
        $8$ & $262\;144$ & $5.95$ & $3.77$ & $5.75$ & $11.70$ & $28.54$ \\
        $16$ & $524\;288$ & $5.94$ & $3.77$ & $5.78$ & $11.69$ & $28.61$ \\
        $32$ & $1\;048\;576$ & $5.96$ & $3.80$ & $5.87$ & $11.73$ & $27.21$ \\
        $64$ & $2\;097\;152$ & $7.92$ & $4.73$ & $6.68$ & $12.62$ & $27.96$ \\
        $128$ & $4\;194\;304$ & $8.03$ & $4.74$ & $6.66$ & $12.76$ & $28.52$ \\
        $256$ & $8\;388\;608$ & $8.09$ & $4.91$ & $6.80$ & $12.90$ & $28.29$ \\
        $512$ & $16\;777\;216$ & $8.06$ & $4.79$ & $6.75$ & $12.82$ & $29.13$ \\
        $1\;024$ & $33\;554\;432$ & $7.96$ & $4.75$ & $6.69$ & $13.01$ & $28.89$ \\
        $2\;048$ & $67\;108\;864$ & $8.06$ & $4.79$ & $6.73$ & $13.02$ & $28.86$ \\
        $4\;096$ & $134\;217\;728$ & $8.14$ & $4.80$ & $6.77$ & $12.77$ & $29.18$ \\
        $8\;192$ & $268\;435\;456$ & $8.14$ & $4.89$ & $6.84$ & $13.03$ & $29.23$ \\
        $16\;384$ & $536\;870\;912$ & $8.10$ & $4.80$ & $6.82$ & $13.25$ & $29.52$ \\
        $32\;768$ & $1\;073\;741\;824$ & $8.21$ & $4.94$ & $6.90$ & $13.19$ & $29.03$ \\ \thickhline
      \end{tabular}
      }\label{chap4_TableWeak}
    }\hfill
    \subfloat[Strong scaling results.]
    {
      \scalebox{0.9}{\begin{tabular}{r|r|r|r|r|r|r}
        \thickhline
        cores & time steps & $p_t=0$ & $p_t =1$ & $p_t=5$ & $p_t=10$ & $p_t=20$ \\ \hline
        $1$ & $1\;048\;576$ & $129.11$ & $78.79$ & $117.99$ & $254.95$ & $535.43$ \\
        $2$ & $1\;048\;576$ & $85.66$ & $54.69$ & $82.36$ & $172.46$ & $396.71$ \\
        $4$ & $1\;048\;576$ & $42.90$ & $27.22$ & $41.17$ & $86.81$ & $199.71$ \\
        $8$ & $1\;048\;576$ & $23.83$ & $15.08$ & $22.90$ & $46.87$ & $107.77$ \\
        $16$ & $1\;048\;576$ & $11.91$ & $7.57$ & $11.50$ & $23.40$ & $54.06$ \\
        $32$ & $1\;048\;576$ & $5.96$ & $3.80$ & $5.87$ & $11.73$ & $27.21$ \\
        $64$ & $1\;048\;576$ & $3.98$ & $2.45$ & $3.31$ & $6.50$ & $13.70$ \\
        $128$ & $1\;048\;576$ & $1.97$ & $1.18$ & $1.67$ & $3.23$ & $6.97$ \\
        $256$ & $1\;048\;576$ & $0.984$ & $0.598$ & $0.808$ & $1.57$ & $3.49$ \\
        $512$ & $1\;048\;576$ & $0.508$ & $0.299$ & $0.407$ & $0.787$ & $1.77$ \\
        $1\;024$ & $1\;048\;576$ & $0.264$ & $0.155$ & $0.210$ & $0.444$ & $0.904$ \\
        $2\;048$ & $1\;048\;576$ & $0.146$ & $0.0864$ & $0.114$ & $0.210$ & $0.465$ \\
        $4\;096$ & $1\;048\;576$ & $0.0861$ & $0.0506$ & $0.0653$ & $0.116$ & $0.243$ \\
        $8\;192$ & $1\;048\;576$ & $0.0548$ & $0.0329$ & $0.0405$ & $0.0743$ & $0.129$ \\
        $16\;384$ & $1\;048\;576$ & $0.0400$ & $0.0230$ & $0.0272$ & $0.0424$ & $0.0767$ \\
        $32\;768$ & $1\;048\;576$ & $0.0511$ & $0.0241$ & $0.0288$ & $0.0376$ & $0.0608$ \\ \thickhline
      \end{tabular}
      }\label{chap4_TableStrong}
    }
    \caption{Scaling results with solving times in seconds.}
\end{table}

%%%%%%%%%%%%%%%%%%%%%%%%%%%%%%%%%%%%%%%%%%%%%%%%%%%%%%%%%%%%%%%%%%%%%%%%%%%%%%%%
%%%%%%%%%%%%%%%%%%%%%%%%%%%%%%%%%%%%%%%%%%%%%%%%%%%%%%%%%%%%%%%%%%%%%%%%%%%%%%%%
\section{Conclusions}\label{ConclusionSec}

We focused in this paper on the analysis of the multigrid method in
time, and our model problem did not contain an operator in space. To
fully leverage the speedup, we consider now the time dependent heat
equation $\partial_t u=\Delta u+f$. Applying our time multigrid
algorithm requires now in each step of the block Jacobi smoother the
solution of Laplace like problems, which we do by just applying one
% \marginpar{Den Vergleich von forward substitution und ST-multigrid habe
% ich auf meinem lokalen Rechner gerechnet und die parallelen Rechnungen auf VSC2)
% - insgesamt sollten aber die Zahlen änliche sein wenn man sie hoch skaliert}
V-cycle of spatial multi-grid. Doing so, we obtain a
  space-time parallel method which takes on one processor for a
  problem of size $131\;120\;896$ a solution time of $10416.90$
  seconds, which is about the same as for forward substitution which
  took $9970.89$ seconds, but which can run in parallel on 2048 cores
  in about $10$ seconds, about one thousand times faster than using
  forward substitution. These results have been computed on the
  Vienna Scientific Cluster VSC-2.
  The precise analysis of this space-time
multigrid algorithm builds on the results we presented in this paper,
but requires techniques for the spatial part that will appear
elsewhere.

%% \begin{table}
%%   \begin{center}
%%       \begin{tabular}{r|r|r|c|r}
%%         \thickhline
%%         cores & time steps & \multicolumn{1}{c|}{dof} & iter & time \\ \hline
%%         $1$ & $4\;096$ & $61\;202\;432$ & $9$ & $6\;960.7$ \\
%%         $2$ & $4\;096$ & $61\;202\;432$ & $9$ & $3\;964.8$ \\
%%         $4$ & $4\;096$ & $61\;202\;432$ & $9$ & $2\;106.2$ \\
%%         $8$ & $4\;096$ & $61\;202\;432$ & $9$ & $1\;056.0$ \\
%%         $16$ & $4\;096$ & $61\;202\;432$ & $9$ & $530.4$ \\
%%         $32$ & $4\;096$ & $61\;202\;432$ & $9$ & $269.5$ \\
%%         $64$ & $4\;096$ & $61\;202\;432$ & $9$ & $135.2$ \\
%%         $128$ & $4\;096$ & $61\;202\;432$ & $9$ & $68.2$ \\
%%         $256$ & $4\;096$ & $61\;202\;432$ & $9$ & $34.7$ \\
%%         $512$ & $4\;096$ & $61\;202\;432$ & $9$ & $17.9$ \\
%%         $1\;024$ & $4\;096$ & $61\;202\;432$ & $9$ & $9.4$ \\
%%         $2\;048$ & $4\;096$ & $61\;202\;432$ & $9$ & $5.4$ \\ \thickhline
%%       \end{tabular}
%%   \end{center}
%%   \caption{Strong scaling results for the heat equation with solving
%%            times in $[s]$.}\label{heatEquStrongScaling}
%% \end{table}
%%%%%%%%%%%%%%%%%%%%%%%%%%%%%%%%%%%%%%%%%%%%%%%%%%%%%%%%%%%%%%%%%%%%%%%%%%%%%%%%
%%%%%%%%%%%%%%%%%%%%%%%%%%%%%%%%%%%%%%%%%%%%%%%%%%%%%%%%%%%%%%%%%%%%%%%%%%%%%%%%

%%%%%%%%%%%%%%%%%%%%%%%%%%%%%%%%%%%%%%%%%%%%%%%%%%%%%%%%%%%%%%%%%%%%%%%%%%%%%%%%
%%%%%%%%%%%%%%%%%%%%%%%%%%%%%%%%%%%%%%%%%%%%%%%%%%%%%%%%%%%%%%%%%%%%%%%%%%%%%%%%
\section*{Acknowledgments}

We thank Ernst Hairer for his help with Theorem \ref{chap4_theorem1a},
and Rolf Krause and Daniel Ruprecht for the simulations we were
allowed to perform on the Monte Rosa supercomputer in Manno.

%%%%%%%%%%%%%%%%%%%%%%%%%%%%%%%%%%%%%%%%%%%%%%%%%%%%%%%%%%%%%%%%%%%%%%%%%%%%%%%%
%%%%%%%%%%%%%%%%%%%%%%%%%%%%%%%%%%%%%%%%%%%%%%%%%%%%%%%%%%%%%%%%%%%%%%%%%%%%%%%%

\bibliographystyle{abbrv}
\bibliography{bibliography}

\begin{thebibliography}{10}

\bibitem{Brandt1977}
A.~{B}randt.
\newblock {M}ulti-level adaptive solutions to boundary-value problems.
\newblock {\em {M}ath. {C}omp.}, 31:333--390, 1977.

\bibitem{Brandt1994}
A.~{B}randt.
\newblock {R}igorous quantitative analysis of multigrid. {I}. {C}onstant
  coefficients two-level cycle with ${L}_2$-norm.
\newblock {\em {SIAM} {J}. {N}umer. {A}nal.}, 31:1695--1730, 1994.

\bibitem{Chipman1971}
F.~{C}hipman.
\newblock {A}-stable {R}unge-{K}utta processes.
\newblock {\em {N}ordisk {T}idskr. {I}nformationsbehandling ({BIT})},
  11:384--388, 1971.

\bibitem{Delfour1981}
M.~{D}elfour, W.~{H}ager, and F.~{T}rochu.
\newblock {D}iscontinuous {G}alerkin methods for ordinary differential
  equations.
\newblock {\em {M}ath. {C}omp.}, 36:455--473, 1981.

\bibitem{Ehle1969}
B.~{E}hle.
\newblock {\em {O}n {P}ad\'{e} approximations to the exponential function and
  {A}-stable methods for the numerical solution of initial value problems}.
\newblock PhD thesis, Technical Report CSRR 2010, Dept. AACS Univ. of Waterloo
  Ontario Canada, 1969.

\bibitem{emmett2012toward}
M.~Emmett and M.~L. Minion.
\newblock Toward an efficient parallel in time method for partial differential
  equations.
\newblock {\em Comm. App. Math. and Comp. Sci}, 7(1):105--132, 2012.

\bibitem{Falgout:2014:PTI}
R.~Falgout, S.~Friedhoff, T.~Kolev, S.~MacLachlan, , and J.~Schr{\"o}der.
\newblock Parallel time integration with multigrid.
\newblock submitted, 2014.

\bibitem{Gander:2014:50Y}
M.~J. Gander.
\newblock 50 years of time parallel time integration.
\newblock In {\em Multiple Shooting and Time Domain Decomposition Methods}.
  Springer Verlag, 2014.

\bibitem{gander2013paraexp}
M.~J. Gander and S.~G{\"u}ttel.
\newblock Para{E}xp: A parallel integrator for linear initial-value problems.
\newblock {\em SIAM Journal on Scientific Computing}, 35(2):C123--C142, 2013.

\bibitem{gander:2008:nca}
M.~J. Gander and E.~Hairer.
\newblock Nonlinear convergence analysis for the parareal algorithm.
\newblock In O.~B. Widlund and D.~E. Keyes, editors, {\em Domain Decomposition
  Methods in Science and Engineering XVII}, volume~60 of {\em Lecture Notes in
  Computational Science and Engineering}, pages 45--56. Springer, 2008.

\bibitem{Gander:2007:OSW}
M.~J. Gander and L.~Halpern.
\newblock Optimized {S}chwarz waveform relaxation methods for advection
  reaction diffusion problems.
\newblock {\em SIAM J. Numer. Anal.}, 45(2):666--697, 2007.

\bibitem{Gander:2014:ADS}
M.~J. Gander and L.~Halpern.
\newblock A direct solver for time parallelization.
\newblock In {\em 22nd international Conference of Domain Decomposition
  Methods}. Springer, 2014.

\bibitem{Gander:2003:OSWW}
M.~J. Gander, L.~Halpern, and F.~Nataf.
\newblock Optimal {S}chwarz waveform relaxation for the one dimensional wave
  equation.
\newblock {\em SIAM Journal of Numerical Analysis}, 41(5):1643--1681, 2003.

\bibitem{Gander:1998:STC}
M.~J. Gander and A.~M. Stuart.
\newblock Space-time continuous analysis of waveform relaxation for the heat
  equation.
\newblock {\em SIAM J. Sci. Comput.}, 19(6):2014--2031, 1998.

\bibitem{gander2007analysis}
M.~J. Gander and S.~Vandewalle.
\newblock Analysis of the parareal time-parallel time-integration method.
\newblock {\em SIAM Journal on Scientific Computing}, 29(2):556--578, 2007.

\bibitem{Hackbusch:1984:PMG}
W.~Hackbusch.
\newblock Parabolic multi-grid methods.
\newblock In R.~Glowinski and J.-L. Lions, editors, {\em Computing Methods in
  Applied Sciences and Engineering, {VI}}, pages 189--197. North-Holland, 1984.

\bibitem{Hairer2010a}
E.~{H}airer, C.~{L}ubich, and G.~{W}anner.
\newblock {\em {G}eometric numerical integration. {S}tructure-preserving
  algorithms for ordinary differential equations}.
\newblock {S}pringer {S}eries in {C}omputational {M}athematics, 31. {S}pringer,
  {H}eidelberg, 2010.

\bibitem{Hairer2010}
E.~{H}airer and G.~{W}anner.
\newblock {\em {S}olving ordinary differential equations. {II}. {S}tiff and
  differential-algebraic problems}.
\newblock {S}pringer {S}eries in {C}omputational {M}athematics, 14.
  {S}pringer-{V}erlag, {B}erlin, 2010.

\bibitem{horton1995space}
G.~Horton and S.~Vandewalle.
\newblock A space-time multigrid method for parabolic partial differential
  equations.
\newblock {\em SIAM Journal on Scientific Computing}, 16(4):848--864, 1995.

\bibitem{Lasaint1974}
P.~{L}asaint and P.-A. {R}aviart.
\newblock {O}n a finite element method for solving the neutron transport
  equation.
\newblock {\em {M}athematical aspects of finite elements in partial
  differential equations ({P}roc. {S}ympos., {M}ath. {R}es. {C}enter, {U}niv.
  {W}isconsin, {M}adison, {W}is., 1974)}, pages 89--123. Publication No. 33,
  Math. Res. Center, Univ. of Wisconsin--Madison, Academic Press, New York,
  1974.

\bibitem{Lions2001}
J.-L. {L}ions, Y.~{M}aday, and G.~{T}urinici.
\newblock {A} "parareal'' in time discretization of {PDE}'s.
\newblock {\em {C}. {R}. {A}cad. {S}ci. {P}aris {S}ér. {I} {M}ath.},
  332:661--668, 2001.

\bibitem{Lubich:1987:MGD}
C.~Lubich and A.~Ostermann.
\newblock Multi-grid dynamic iteration for parabolic equations.
\newblock {\em BIT}, 27(2):216--234, 1987.

\bibitem{maday2008parallelization}
Y.~Maday and E.~M. R{\o}nquist.
\newblock Parallelization in time through tensor-product space--time solvers.
\newblock {\em Comptes Rendus Mathematique}, 346(1):113--118, 2008.

\bibitem{NeumuellerThesis2013}
M.~Neum{\"u}ller.
\newblock {\em Space-Time Methods: Fast Solvers and Applications}.
\newblock PhD thesis, University of Graz, 2013.

\bibitem{Reed1973}
W.~{R}eed and T.~{H}ill.
\newblock {T}riangular mesh methods for the neutron transport equation.
\newblock {\em {T}ech {R}eport {LAUR}73479 {L}os {A}lamos {N}ational
  {L}aboratory}, Technical, Issue: LA-UR-73-479:1--23, 1973.

\bibitem{speck2013multi}
R.~Speck, D.~Ruprecht, M.~Emmett, M.~Minion, M.~Bolten, and R.~Krause.
\newblock A multi-level spectral deferred correction method.
\newblock {\em arXiv preprint arXiv:1307.1312}, 2013.

\bibitem{speck2012massively}
R.~Speck, D.~Ruprecht, R.~Krause, M.~Emmett, M.~Minion, M.~Winkel, and
  P.~Gibbon.
\newblock A massively space-time parallel n-body solver.
\newblock In {\em Proceedings of the International Conference on High
  Performance Computing, Networking, Storage and Analysis}, page~92. IEEE
  Computer Society Press, 2012.

\bibitem{Stueben1985}
K.~{S}t\"{u}ben and U.~{T}rottenberg.
\newblock {\em {M}ultigrid methods: fundamental algorithms, model problem
  analysis and applications}.
\newblock {GMD}-{S}tudien [{GMD} {S}tudies], 96. {G}esellschaft f\"{u}r
  {M}athematik und {D}atenverarbeitung mb{H}, {S}t. {A}ugustin, 1985.

\bibitem{Trottenberg2001}
U.~{T}rottenberg, C.~W. {O}osterlee, and A.~{S}ch\"{u}ller.
\newblock {\em {M}ultigrid}.
\newblock {A}cademic {P}ress, {I}nc., {S}an {D}iego, 2001.

\bibitem{Wesseling2004}
P.~{W}esseling.
\newblock {\em {A}n {I}ntroduction to {M}ultigrid {M}ethods}.
\newblock {J}ohn {W}iley \& {S}ons {L}td., 1992. {C}orrected {R}eprint.
  {P}hiladelphia: {R}.{T}. {E}dwards, {I}nc., 2004.

\end{thebibliography}

\end{document}